\documentclass[twocolumn]{autart}    

\usepackage{graphicx}          

\usepackage{amsmath,amssymb,amsfonts}  

\usepackage{amssymb} 

\usepackage[T1]{fontenc}  

\usepackage{subfigure}  

\usepackage{xcolor} 

\usepackage[numbers]{natbib} 

\usepackage{nccmath}  

\begin{document}

\begin{frontmatter}

\title{Decentralized Strategies for Finite Population LQG Social Control: A Reinforcement Learning Approach \thanksref{footnoteinfo}} 

\thanks[footnoteinfo]{Corresponding author: Bing-Chang Wang.}
\thanks[footnoteinfo]{  This research was supported by the National Key R\&D Program of China under Grant No.2022YFA1006100, 
the National Natural Science Foundation of China under Grant Nos.62573266 and 61925306.}

\author{Liangyuan Guo}\ead{ lyguo@mail.sdu.edu.cn },    
\author{Bing-Chang Wang}\ead{bcwang@sdu.edu.cn},               
\author{Guangchen Wang}\ead{ wguangchen@sdu.edu.cn}  

\address{School of Control Science and Engineering, Shandong University, Jinan 250061, China}  


\begin{abstract}                          
    This paper presents a novel model-free algorithm for the finite-population  linear quadratic Gaussian (LQG) decentralized social control problem  with multiplicative noise.
    The state and control weights in the cost functional are not limited to be positive semidefinite.
    For both finite-horizon and infinite-horizon cases, the goal is to  obtain a social optimum by solving two algebraic Riccati equations (AREs), without requiring prior knowledge of the system matrices.
    Then, we complete the design of a model-free algorithm for solving the decentralized social control problem. 
    Especially, in the infinite-horizon case, the algorithm's convergence is based on analyzing the spectral property of the Lyapunov-type operator.
    The differences of reinforcement learning (RL) solutions between the finite-horizon and infinite-horizon cases are compared. 
    Finally,  the effectiveness of the proposed algorithm is demonstrated by a numerical example.

\end{abstract}

\end{frontmatter}

\section{Introduction}

The decentralized social control provides a useful method for solving large population cooperative differential games, which are characterized by agents seeking to optimize a common cost.
Decentralized social control has attracted considerable attention in various fields,  
including system control, financial mathematics, crowd dynamics and engineering  (\cite{aurell2018mean}, \cite{cardaliaguet2019master}, \cite{carmona2013control}, \cite{tembine2013risk}).
 
\subsection{Mean field game and control}

Mean field games originated from the pioneering achievements of \cite{Huang2007Large} and \cite{Lasry2007MeanFG}. 
The mean field game and control is full of potential in theoretical and practical application. 
A mean field game can be categorized into two types, namely the linear quadratic Gaussian (LQG) type and the more general type, 
which is determined by its state-cost setup. 
The LQG type is prevalently used in mean field problems, due to its wide practical backgrouds.
See \cite{Moon2017Linear} for LQG type mean field games with risk-sensitive and worst-case risk-neutral cost, 
\cite{Bardi2014Linear} with ergodic cost, \cite{Li2008Asymptotically} with time-averaged stochastic cost, 
\cite{Wang2012Mean} with Markov jump parameters, 
\cite{Li2025Linear} with partial observation.
Some recent works for nonlinear mean field games include \cite{Fischer2017ON} and \cite{Lacker2020ON}. 
In addition, the robust mean field games are studied in \cite{Huang2017Robust}, 
Stackelberg mean field games in \cite{Bensoussan2015Mean} and \cite{Li2024LinearQuadraticSS}, zero-sum differential mean field games in \cite{Moon2020Linear}.

Mean field models have been extensively studied for social optima. 
Social optima are linked to a type of team decision problem \cite{Ho1980Decision} but with highly complex interactions. 
Accordingly, a team decision problem can be formulated where all cooperative agents share a common cost functional. 
The goal of collaborating agents is to minimize the total cost in a large population system.
However, achieving social optima often requires centralized information from all agents, which poses huge computational and implementation challenges. 
To cope with these challenges, decentralized cooperative optimization frameworks have been developed, utilizing the mean field model to asymptotically achieve social optima.
\cite{Huang2007Large} developed a state aggregation technique  to obtain a set of decentralized control laws for the individuals which possesses an $\epsilon$-Nash equilibrium property.
\cite{Huang2012Social} developed the social certainty equivalence approach for asymptotically achieving the social optima. 
\cite{Wang2020Mean} gave the necessary and sufficient conditions for uniform stabilization of the systems
and designed a set of decentralized control laws by using solutions of two Riccati equations.
\cite{Huang2021Linear} employed a rescaling approach to reduce the Riccati ordinary differential equation (ODE) to some lower-order ODE system, 
which characterizes a necessary and sufficient condition for asymptotic solvability.

Most existing works on mean-field games and control have studied the large or infinite population systems.
However, small or moderate-population systems are also considered in many practical situations \cite{Ma2013Decentralized}, \cite{Kizilkale2016load}.
 In mean-field games and control theory, the classical results have limitations for finite-population systems. 
Specifically, the asymptotically optimal strategies derived from these results may lead to performance loss in finite-population systems.
Therefore, it is significant to study the optimal decentralized control for the finite-population systems.
\cite{Wang2023Decentralized} first gave necessary and sufficient conditions for the solvability of finite-population decentralized games and teams in terms of the  forward-backward stochastic difference equations (FBSDEs).
\cite{Liang2024Discrete} considered the optimal decentralized strategies for finite-population systems in discrete-time settings.
Notably, the cost function's weighting matrices are not required to be definite 
and the solvability of decentralized social optimal control is characterized by the existence of
regular solutions to two Riccati equations.

\subsection{Reinforcement learning}

In recent decades, the numerical methods for the mean field games and control have attracted tremendous attention.
Several numerical methods for mean-field models have been designed, 
such as the fixed-point approach (\cite{Caines2021Graphon}, \cite{Li2008Decentralized}), 
Semi-Lagrangian scheme (\cite{Carlini2012AFD}), finite difference discretization (\cite{Arias2018Proximal}), etc.
However, these methods require the system dynamics to be known, which limits their practical applicability.
 
The problem of learning solutions to mean field game and control without requiring the system matrices has garnered significant interest. 
Recently, Reinforcement learning (RL) has been successfully applied to solve complex decision making problems \cite{wang2020reinforcement}.
In RL, a agent learns the control policies by interacting with the environment \cite{Rizvi2018Output}. 
For more details, see \cite{sutton1998reinforcement}. 
Various algorithms have been developed for mean field games and control, 
including  Q-learning (\cite{guo2019learning}, \cite{angiuli2022unified}, \cite{gu2021mean}), 
entropy regularization (\cite{Guo2022Entropy}), fictitious play (\cite{perrin2020fictitious}). 
The connection between RL for mean field control and finite-agent problems is analyzed in \cite{wang2020breaking}, \cite{chen2021pessimism}.
Nonhomogeneous populations are considered in \cite{hu2023graphon}. 
\cite{Li2024Policy} employed a policy iteration RL method using a Lyapunov Recursion for the continuous-time LQG
mean-field control problems in infinite horizon.
\cite{Xu2023Model} and \cite{XU2025Mean} presented a model-free method to solve LQG mean field games and mean field social control, respectively.
However, most studies focus either on mean field games or on infinite population systems.
Research on the learning algorithms for mean field social control on finite population is relatively scarce.

\subsection{Contribution} 

Most of related work focuses on the infinite-population or large-population systems with traditional numerical methods.
However, the existing research has the following gaps. 
(1) Results on classical mean field games and control mainly focused on large-population or infinite-population systems.
However, for the practical finite-population systems, using strategies based on the above results can cause a considerable optimality loss.
The corresponding data-driven methods do not work well for finite-population systems.
(2) Existing model-free algorithms predominantly focus on the infinite-horizon case and the work on finite-horizon problems is scarce.
The relationship of RL solutions under the above two cases  remains unclear,
necessitating further investigation to bridge this gap.
(3) Existing results on the convergence of policy iteration algorithms are guaranteed 
when the weight matrices $Q$ and $R$ satisfy (semi-) definiteness conditions.
Conversely, for non-standard/indefinite weight matrices, convergence analyses are sparse.

In this paper, we study the finite-population LQG mean field social control problems in the discrete time.
If the system dynamics are known, the optimal control gains can be obtained by solving two Riccati equations \cite{Liang2024Discrete}.
However, the technical challenge arises when the dynamics of agents are completely unknown.
Inspired by the abovementioned related work, especially \cite{Liang2024Discrete} and \cite{XU2025Mean}, 
this paper develops a model-free algorithm for solving the finite-population LQG mean field social control problem.

{ Different from existing work, this paper considers the model-free learning for  finite-population systems.}
The main contributions of the paper are summarized as follows.

(1) We develop a model-free algorithm to find decentralized strategies under indefinite weight matrices. 
Unlike traditional methods that rely on positive (semi-) definite state and control weights for convergence, 
we relax this condition to allow for indefinite weight matrices.
Furthermore, we demonstrate three differences in RL solutions between finite-horizon and infinite-horizon cases.

(2) We present the monotonic decrease property for iterative equations, which is different from the matrix inversion method.
When the quadratic term related to control ($B^T \bar{P} B$) exists in the ARE, the direct matrix inversion formula method is infeasible.
We combine two coupled AREs into a new equation and prove the monotonic decrease and boundedness of the iterative algorithms, ensuring the convergence of the iterative sequences.

(3) For addressing the indefinite AREs, 
we show the convergence of iterative algorithms with indefinite weights for both finite-horizon and infinite-horizon cases.
Notably, 
in the infinite-horizon case, it builds upon analyzing the spectral property of the Lyapunov-type operator.

\subsection{Notations}
Let $\mathbb{R}^{m \times n}$ be the set of all $m \times n$ matrices.
A matrix $A \in \mathbb{R}^{n \times n}$ is positive definite (resp. positive semi-definite), expressed as $A \in \mathbb{S}^n_+$ (resp. $A \in \mathbb{S}^n$).
$\operatorname{vecs}(A) \triangleq [a_{11},2a_{12},\ldots,2a_{1n},a_{22}, \ldots, a_{nn}]^{\mathrm{T}}  \in\mathbb{R}^{\frac12n(n+1)}.$
$\otimes$ is the Kronecker product.
For a vector $x \in \mathbb{R}^n$ and a matrix $A \in \mathbb{R}^{n \times m}$, 
define the following operators: 
$\tilde{x} \triangleq [x_{1}x_{1}, \dots,x_{1}x_{n},x_{2}x_{2}, \dots,x_{2}x_{n},\dots,x_{n}x_{n}] \in\mathbb{R}^{\frac12n(n+1)}.$
$\mathrm{vec}(A) \triangleq [ a_{1}^{\mathrm{T}}, a_{2}^{\mathrm{T}},\ldots, a_{m}^{\mathrm{T}} ]^{\mathrm{T}}\in\mathbb{R}^ {nm},$ 
where $a_{i}$ is the $i$-th column of $A$. $diag(\cdots)$ is the diagonal matrix.

\section{Problem Formulation}
Consider a finite-population system with $N$ homogeneous players, 
where the dynamics of player $\mathcal{A}_{j}$ are given by
the discrete-time linear stochastic difference equation
\begin{equation}
x^j(k+1)=Ax^j(k)+Bu^j(k)+(Cx^j(k)+Du^j(k))w^j(k), \label{e1_1}
\end{equation}
where $\{x_k^j \in \mathbb{R}^n,k \in \mathbb{T}\}$ and $\{u_k^j \in \mathbb{R}^m,k \in \mathbb{T}\}$ 
are the state and control processes of the player $\mathcal{A}_{j}$, respectively.
$w^j(k)$ is the scalar Gaussian  white noise with zero mean and variance $\sigma^2$.
$A,C \in \mathbb{R}^{n \times n}$ and $B,D \in \mathbb{R}^{n \times m}$, the
coefficient matrices are given deterministic.
The individual cost function of player $\mathcal{A}_{j}$ is given by
\begin{equation}
    \begin{aligned}
        J_i(u)=&\ \sum_{k=0}^T \mathbb{E}\Big( \big| x^j_k-\Gamma x^{(N)}_k\big|^2_{Q}+\big|u^j_k\big|^2_{R}\Big)\\
        &+ \mathbb{E} \big|x^j_{T+1}-\Gamma x^{(N)}_{T+1}\big|^2_{Q}, \nonumber
    \end{aligned}
\end{equation}
where $u \triangleq ((u^1)^T, \dots, (u^N)^T)^T$,
and  
$Q, R, \Gamma$ are constant matrices with appropriate dimensions. 
$Q$ and $R$ are not limited to be positive semi-definite. 
$x^{(N)}_k=\frac{1}{N}\sum_{j=1}^N x^j_k$ is the mean field term.
The social cost function is given by 
\begin{equation}
    J_{\mathrm{soc}} (u) = \sum_{j=1}^NJ_j(u). \label{e1_2}
\end{equation} 
The admissible decentralized control set for the player $\mathcal{A}_{j}$ is given by
\begin{equation}
    \mathcal{U}_{ad}^j=\{u^j|u_k^j \; \text{is adapted to} \; \mathcal{F}_{k}^j,  \; \sum_{k=0}^T \mathbb{E}|u_k^j|^2<\infty\},  \label{e1_3}
\end{equation}
where $\mathcal{F}_k^j$ is the $\sigma$-field generated by $\{w_0^j,w_1^j,\ldots,w_{k-1}^j\}$.
A set of strategies $\{\hat{u}^{j}, 1\leq j \leq N \}$ is said to be a
decentralized social optimal control if the following holds: 
\begin{equation}
    J_{\text{soc}} (\hat{u})=\inf_{u^j\in\mathcal{U}_{ad}^j,j\in\mathbb{N}}J_{\text{soc}} (u), \label{e1_4}
\end{equation}
where 
$\hat{u} \triangleq ((\hat{u}^1)^T, \dots, $ $(\hat{u}^N)^T)^T.$ 
Besides, define the individual cost function  in infinite horizon  
    \begin{equation}
        \begin{aligned}
            J_j^{\infty}(u^j,u^{-j})=&\ \sum_{k=0}^\infty \mathbb{E}\Big( \big| x^j_k-\Gamma x^{(N)}_k\big|^2_{Q}+\big|u^j_k\big|^2_{R}\Big).  \label{problem_infinite}
        \end{aligned} 
    \end{equation}

In this paper, we mainly study the following problems.

\begin{prob}
    (Finite-horizon problem)
    For each player $\mathcal{A}_{i}$ with the state equation \eqref{e1_1} associated the   cost function \eqref{e1_2}, 
    find a decentralized social optimal control $\{\hat{u}^{j}, 1\leq j \leq N  \}$ 
    such that \eqref{e1_4} holds for any $x_{j0}\in\mathbb{R}^{n}$, without knowing players' dynamics (Parameters $A,B,C,D$ are unknown).
\end{prob}

\begin{prob}
    (Infinite-horizon problem) 
    For each player $\mathcal{A}_{i}$ with the state equation \eqref{e1_1} associated the   cost function $J_{\mathrm{soc}}^{\infty  }(u) = \sum_{j=1}^NJ_j^{\infty}(u^j,u^{-j})$, 
    the goal is to 
     find a decentralized social optimal control $\{\hat{u}^{j}, 1\leq j \leq N  \}$ 
    to optimize the cost function, without knowing players' dynamics.  
\end{prob}

The related concepts of stabilizability (as defined in \cite{ni2015indefinite}) are introduced as follows.
We first consider the following system:
\begin{equation}
\begin{aligned}
     y(k&+1)=  (A+BK) y(k) + (B\bar{K}-BK) \mathbb{E}(y(k)) \\
     & + [(C+DK) y(k) + (D\bar{K}-DK) \mathbb{E}(y(k))]  w(k).
\end{aligned}    
 \label{e_def_1}
\end{equation} 

\begin{defn}
    System \eqref{e_def_1} is called closed-loop $L^2-$ stabilizable  if there exists a pair $(K,\bar{K})$
    such that for any initial state $x_0$, 
    the closed loop of system \eqref{e_def_1} is $L^2-$stable.
\end{defn}

\section{Main results for Problem 1}
In this section, we  develop a model-free method to 
approximate the solution of AREs \eqref{e1_5} and \eqref{e1_6}.
First, we design two model-based policy iteration algorithms to solve AREs \eqref{e1_5} and \eqref{e1_6}.
Then, by employing the model-free method, the reliance of system dynamics can be avoided.

\subsection{Model-based policy iteration}

We now introduce the following definition and lemma, which can be referred to \cite{Liang2024Discrete}.

\begin{defn} \label{linearly_solvable}
    Problem 1 is linearly solvable if Problem 1 admits a decentralized social optimal control with the linear structure
    $u_k^j = K_{k} x_k^j + \bar{K}_k \mathbb{E}[x_k^j] + \eta_k.$
\end{defn}

\begin{lem}  \cite[Theorem 4.5]{Liang2024Discrete}
    Problem 1 is linearly solvable if and only if 
    AREs \eqref{e1_5}-\eqref{e1_6} admit solutions $P_k$, $\bar{P_k}$
    such that the regular conditions \eqref{e1_regular_conditions} hold. 
   In this case, the decentralized social optimal control can be represented as
    \begin{equation}
        \begin{aligned} 
        u_k^j = -  {\Upsilon}_{k}^{-1} M_{k} (x_k^j-\bar{x}_k^j)    
        -  \bar{\Upsilon}_{k}^{-1} \bar{M}_{k}  \bar{x}_k^j, 
         \end{aligned}
    \end{equation}
    where
    \begin{equation}
        \begin{aligned}
 &    {\Upsilon}_{k} \triangleq R  + B^T  {P}_{k+1} B + \sigma^2 D^T  {P}_{k+1} D, \\
 &    M_{k}  \triangleq B^T  {P}_{k+1} A + \sigma^2 D^T  {P}_{k+1} C,\\
 &    \bar{\Upsilon}_{k} \triangleq \Upsilon_{k} + B^T \bar{P}_{k+1} B,\\
 &    \bar{M}_{k} \triangleq M_{k} + B^T \bar{P}_{k+1} A.
        \end{aligned} \nonumber
    \end{equation}
     
\begin{equation}\medmath{
\left\{
\begin{aligned}
    &P_{k}=Q+\frac{1}{N}(\Gamma^TQ\Gamma-Q\Gamma-\Gamma^TQ)+A^TP_{k+1}A \\
    &+\sigma^2C^TP_{k+1}C  -(B^TP_{k+1}A+\sigma^2D^TP_{k+1}C)^T \\
    &\times   (R+B^TP_{k+1}B+\sigma^2D^TP_{k+1}D)^{-1} \\
    &\times  (B^TP_{k+1}A+\sigma^2D^TP_{k+1}C), \\
    &P_{T+1}=Q+\frac{1}{N}(\Gamma^TQ\Gamma-Q\Gamma-\Gamma^TQ).
\end{aligned}    
\right.} \label{e1_5}
\end{equation}
\begin{equation} \medmath{  
\left\{     
\begin{aligned} 
    &\bar{P_{k}}=\frac{N-1}{N}(\Gamma^TQ\Gamma-Q\Gamma-\Gamma^TQ)+A^T\bar{P}_{k+1}A \\
    &+(B^TP_{k+1}A+\sigma^2D^TP_{k+1}C)^T(R+B^TP_{k+1}B \\
    &+\sigma^2D^TP_{k+1}D)^{-1}(B^TP_{k+1}A+\sigma^2D^TP_{k+1}C)\\
    &-(B^TP_{k+1}A+\sigma^2D^TP_{k+1}C+B^T\bar{P}_{k+1}A)^T \\
    &\times (R+B^TP_{k+1}B+\sigma^2D^TP_{k+1}D+B^T\bar{P}_{k+1}B)^{-1} \\
    &\times  (B^TP_{k+1}A+\sigma^2D^TP_{k+1}C+B^T\bar{P}_{k+1}A), \\
    &\bar{P}_{T+1}= \frac{N-1}{N}(\Gamma^TQ\Gamma-Q\Gamma-\Gamma^TQ).
\end{aligned}
\right.} \label{e1_6}
\end{equation}
\begin{equation} \medmath{ 
\left\{
\begin{aligned}
    & {\Upsilon}_{k} \geq 0, \\
    & \bar{\Upsilon}_{k} \geq 0, \\
    & [I - {\Upsilon}_{k} {\Upsilon}_{k}^{T}] M_{k}=0, \\
    & [I - \bar{\Upsilon}_{k} \bar{\Upsilon}_{k}^{T}] \bar{M}_{k}=0.
\end{aligned}    
\right.} \label{e1_regular_conditions}
\end{equation} 
\end{lem}

We now make the following assumption.
\begin{assum} \label{ass_are_solution}
    The AREs \eqref{e1_5}-\eqref{e1_6} admit a regular solution.
\end{assum}

For notational simplicity, let
$K_{k} \triangleq -{\Upsilon}_{k}^{-1}M_{k}$,
$\bar{K}_{k} \triangleq -\bar{\Upsilon}_{k}^{-1} \bar{M}_{k}$.

\begin{thm} \label{thm1}
    Assume that the ARE \eqref{e1_5} exists a regular solution.
    $P_{k}^i$ and $K_{k}^{i+1}$ are recursively updated by \eqref{th1_1}-\eqref{th1_2}, 
    \begin{equation}
        \begin{aligned} 
            P_{k}^i=& Q+\frac{1}{N}(\Gamma^TQ\Gamma-Q\Gamma-\Gamma^TQ)
            +{(A_{k}^i)}^TP_{k+1}^iA_{k}^i \\
            & +\sigma^2{(C_{k}^i)}^TP_{k+1}^iC_{k}^i+{(K_{k}^i)}^TRK_{k}^i, 
        \end{aligned} \label{th1_1}
    \end{equation}
    \begin{equation}
        \begin{aligned} 
            K_{k}^{i+1}=& -(R+B^TP_{k+1}^iB+\sigma^2 D^TP_{k+1}^iD)^{-1} \\
            &\times  (B^TP_{k+1}^iA+\sigma^2 D^TP_{k+1}^iC),
        \end{aligned} \label{th1_2}
    \end{equation}
    where 
    $A_k^i \triangleq A+BK_k^i$, $C_k^i \triangleq C+DK_k^i$,
    and $P^i_{T+1} \equiv  Q+\frac{1}{N}(\Gamma^TQ\Gamma-Q\Gamma-\Gamma^TQ)$. 
    Then  for any 
$k  \in [0,T]$,
$P_{k} \leq P_{k}^{i+1} \leq P_{k}^i$, 
$\lim\limits_{i\to+\infty} P_{k}^i = P_{k}$.
\end{thm}
\emph{Proof.} Since
$A_{k}^{i}=A_{k}^{i+1}+B(K_{k}^{i}-K_{k}^{i+1})$,
$C_{k}^{i}=C_{k}^{i+1}+D(K_{k}^{i}-K_{k}^{i+1})$,
$K_{k}^{i}=(K_{k}^{i}-K_{k}^{i+1})+K_{k}^{i+1}$,
we have
\begin{equation}
    \begin{aligned} 
        P_{k}^{i}-P_{k}^{i+1}=&(K_{k}^{i}-K_{k}^{i+1})^T (R+B^TP_{k+1}^iB \\
        & +\sigma^2 D^TP_{k+1}^iD)(K_{k}^{i}-K_{k}^{i+1}) \\ 
        &+{A_{k}^{i+1}}^T(P_{k+1}^{i}-P_{k+1}^{i+1})A_{k}^{i+1} \\
        &+\sigma^2{C_{k}^{i+1}}^T(P_{k+1}^{i}-P_{k+1}^{i+1})C_{k}^{i+1}. 
    \end{aligned} \label{pf1_1}
\end{equation}
Noting that 
$P_{T+1}^{i}=P_{T+1}^{i+1}=Q+\frac{1}{N}(\Gamma^TQ\Gamma-Q\Gamma-\Gamma^TQ)$
$R+B^TP_{k+1}^iB+\sigma^2 D^TP_{k+1}^iD \geq 0$, 
one has for all $i,k \in \mathbb{N}^{+}$, 
$P_{k+1}^{i}-P_{k+1}^{i+1} \geq 0$.
Combining \eqref{th1_1} with \eqref{e1_5}, we have
\begin{equation}
    \begin{aligned} 
        P_{k}^{i+1}-P_{k}=&(K_{k}-K_{k}^{i+1})^T (R+B^TP_{k+1}B \\
        & +\sigma^2 D^TP_{k+1}D)(K_{k}-K_{k}^{i+1}) \\ 
        &+{A_{k}^{i+1}}^T(P_{k+1}^{i+1}-P_{k+1})A_{k}^{i+1} \\
        &+\sigma^2{C_{k}^{i+1}}^T(P_{k+1}^{i+1}-P_{k+1})C_{k}^{i+1}. 
    \end{aligned} \label{pf1_1_2}
\end{equation}
Therefore, the inequality $P_{k}^{i+1}-P_{k} \geq 0$ holds. 
This implies that the sequence $\{ P_{k}^i \}_{i=0}^{\infty}$ is monotonically decreasing and has a lower bound. 
The limit of the sequence $\{ P_{k}^i \}_{i=0}^{\infty}$ exists.

Letting $i \to \infty$ in \eqref{th1_1} and \eqref{th1_2}, we have
\begin{equation}
    \begin{aligned} 
        P_{k}^{\infty}=&Q+\frac{1}{N}(\Gamma^TQ\Gamma-Q\Gamma-\Gamma^TQ) \\
        &+{A_{k}^{\infty}}^TP_{k+1}^{\infty}A_{k}^{\infty}+{K_{k}^{\infty}}^TRK_{k}^{\infty} \\
        &+\sigma^2{C_{k}^{\infty}}^TP_{k+1}^{\infty}C_{k}^{\infty},
    \end{aligned} \label{pf1_2}
\end{equation}
\begin{equation}
    \begin{aligned} 
        K_{k}^{\infty}=-(R+B^TP_{k+1}^{\infty}B+\sigma^2 D^TP_{k+1}^{\infty}D)^{-1} \\
        \times (B^TP_{k+1}^{\infty}A+\sigma^2 D^TP_{k+1}^{\infty}C).
    \end{aligned} \label{pf1_3}
\end{equation}
Eliminating $K_{k}^{\infty}$ from the above equations, one has
\begin{equation}
    \begin{aligned} 
             P_{k}^{\infty}=&Q+\frac{1}{N}(\Gamma^TQ\Gamma-Q\Gamma-\Gamma^TQ) \\
            & +{A}^TP_{k+1}^{\infty}A+\sigma^2{C}^TP_{k+1}^{\infty}C \\
            & -(B^TP_{k+1}^{\infty}A
            +\sigma^2D^TP_{k+1}^{\infty}C)^T \\ 
            &\times  (R+B^TP_{k+1}^{\infty}B
            +\sigma^2D^TP_{k+1}^{\infty}D)^{-1} \\
            &\times  (B^TP_{k+1}^{\infty}A+\sigma^2D^TP_{k+1}^{\infty}C).      
    \end{aligned} \label{pf1_4}
\end{equation}
Since 
$P_{T+1}^{\infty}=P_{T+1}=Q+\frac{1}{N}(\Gamma^TQ\Gamma-Q\Gamma-\Gamma^TQ)$
and the ARE \eqref{e1_5} has a unique positive semi-definite solution, 
we have 
$P_{k}^{\infty}=P_{k}$.
{{\hfill{$\Box$}}}

\begin{rem}
    To obtain the optimal control gain and the solution of the ARE \eqref{e1_5},
   the iteration \eqref{th1_1}-\eqref{th1_2} need complete knowledge of the system dynamics. 
   When the system is a single-agent and deterministic system, 
   the above iteration degenerates into the method in 
   \cite{Lewis2011Reinforcement}.
\end{rem}

\begin{rem}
    Most existing numerical methods are developed for continuous-time, infinite-horizon or large-population systems.
    For other situations, 
    it remains uncertain how to transform the model-based approach into model-free one. 
    Based on Theorem \ref{thm1}, a corresponding model-free off-policy algorithm can be deduced.
\end{rem}

Combining \eqref{e1_5} with \eqref{e1_6}, one has
\begin{equation} 
\begin{aligned} 
S_k =& Q + \Gamma ^ T Q \Gamma - Q \Gamma - \Gamma ^ T Q + \bar{A}_k^T S_{k+1} \bar{A}_k \\
& + \sigma^2 \bar{C}_k^T P_{k+1} \bar{C}_k + \bar{K}_k^T R \bar{K}_k,
\end{aligned}   \label{e1_7}
\end{equation}
where 
$S_{k} \triangleq P_k + \bar{P}_k$, $\bar{A}_k \triangleq A+B \bar{K}_k$ and $\bar{C}_k \triangleq C+D \bar{K}_k$.

\begin{thm} \label{thm2}
    Assume that AREs \eqref{e1_5}-\eqref{e1_6} exist a regular solution.
    $S_k^i$ and $\bar{K}_{k}^{i+1}$ are recursively updated by \eqref{th2_1}-\eqref{th2_2}, 
\begin{equation}   
\begin{aligned}
S_k^i =& Q + \Gamma ^ T Q \Gamma - Q \Gamma - \Gamma ^ T Q + {{}(\bar{A}^i_k)}^{T} S_{k+1}^i {\bar{A}_k^i} \\
&+ \sigma^2 {{}(\bar{C}^i_k)}^{T} P_{k+1} \bar{C}_k^i + {{}(\bar{K}^i_k)}^{T} R \bar{K}_k^i, 
\end{aligned}   \label{th2_1}
\end{equation}
\begin{equation}   
\begin{aligned}
\bar{K}_k^{i+1} = - R^{-1} (B^TS_{k+1}^i  {\bar{A}_k^{i+1}} + \sigma^2 D^T P_{k+1} {\bar{C}_k^{i+1}}),
\end{aligned}   \label{th2_2}
\end{equation}
where 
$S^i_{T+1}\equiv  Q+ \Gamma^TQ\Gamma-Q\Gamma-\Gamma^TQ $,
and 
$P_k$ is the converged solution obtained from Theorem \ref{thm1} .
Then  for any 
$k  \in [0,T]$,
$S_{k} \leq S_{k}^{i+1} \leq S_{k}^i$, 
$\lim\limits_{i\to+\infty} S_{k}^i = S_{k}$ and  $\lim\limits_{i\to+\infty} \bar{P}^i_k = \bar{P}_k$.
\end{thm}
\emph{Proof.} Since  
$\bar{A}_{k}^{i}=\bar{A}_{k}^{i+1}+B(\bar{K}_{k}^{i}-\bar{K}_{k}^{i+1})$,
$\bar{C}_{k}^{i}=\bar{C}_{k}^{i+1}+D(\bar{K}_{k}^{i}-\bar{K}_{k}^{i+1})$,
$\bar{K}_{k}^{i}=(\bar{K}_{k}^{i}-\bar{K}_{k}^{i+1})+\bar{K}_{k}^{i+1}$, 
we have
\begin{equation}  
\begin{aligned} 
    S_{k}^{i}&-S_{k}^{i+1}={{}(\bar{A}_{k}^{i})}^T S_{k+1}^{i} \bar{A}_{k}^{i} -{{}(\bar{A}_{k}^{i+1})}^T S_{k+1}^{i+1} \bar{A}_{k}^{i+1} \\
    &+\sigma^2{{}(\bar{C}_{k}^{i})}^T P_{k+1} \bar{C}_{k}^{i} - \sigma^2{{}(\bar{C}_{k}^{i+1})}^T P_{k+1} \bar{C}_{k}^{i+1} \\
    &+{(K_{k}^{i})}^T R {K_{k}^{i}} - {(K_{k}^{i+1})}^T R {K_{k}^{i+1}}.
\end{aligned}    \label{pf2_1}
\end{equation}
Combining the above equation with \eqref{th2_2}, one has
\begin{equation} 
\begin{aligned} 
    S_{k}^{i}-S_{k}^{i+1}=& {{}(\bar{A}_{k}^{i+1})}^T (S_{k+1}^{i} - S_{k+1}^{i+1}) \bar{A}_{k}^{i+1} \\
    &+ ({\bar{K}_{k}^{i}}-{\bar{K}_{k}^{i+1}})^T (R+B^TS_{k+1}^i B \\
    &+ \sigma^2 D^TP_{k+1}D) ({\bar{K}_{k}^{i}}-{\bar{K}_{k}^{i+1}}). 
\end{aligned}  \label{pf2_2}
\end{equation}
Noting that 
$S_{T+1}^{i}=S_{T+1}^{i+1}=P_{T+1}+\bar{P}_{T+1}= Q+(\Gamma^TQ\Gamma-Q\Gamma-\Gamma^TQ)$,
$R+B^TS_{k+1}^i B + \sigma^2 D^TP_{k+1}D \geq 0$, 
one has for all $i,k \in \mathbb{N}^{+}$,
$S_{k}^{i}-S_{k}^{i+1} \geq 0$.
Combining \eqref{e1_5}, \eqref{e1_6} with \eqref{th2_1}, we have
    \begin{equation} 
    \begin{aligned} 
        S_{k}^{i+1}-S_{k}=& {{}(\bar{A}_{k}^{i+1})}^T (S_{k+1}^{i+1} - S_{k+1}) \bar{A}_{k}^{i+1} \\
        &+ ({\bar{K}_{k}}-{\bar{K}_{k}^{i+1}})^T (R+B^TS_{k+1} B \\
        &+ \sigma^2 D^TP_{k+1}D) ({\bar{K}_{k}}-{\bar{K}_{k}^{i+1}}). 
    \end{aligned}  \label{pf2_2_1}
\end{equation}
Therefore, the inequality $S_{k}^{i+1}-S_{k} \geq 0$ holds. 
It implies that the sequence $\{ S_{k}^i \}_{i=0}^{\infty}$ is monotonically decreasing and has a lower bound. 
The limit of the sequence $\{ S_{k}^i \}_{i=0}^{\infty}$ exists.
Letting $i \to \infty$ in \eqref{th2_1} and \eqref{th2_2}, one has
\begin{equation} 
    \begin{aligned} 
        S_k^\infty =& Q + \Gamma ^ T Q \Gamma - Q \Gamma - \Gamma ^ T Q + {{}\bar{A}_k^\infty}^T S_{k+1} \bar{A}_k^\infty \\
        &+ \sigma^2 {{}\bar{C}_k^\infty}^T P_{k+1} \bar{C}_k^\infty + {{}\bar{K}_k^\infty}^T R \bar{K}_k^\infty. 
\end{aligned}  \label{pf2_3}
\end{equation}
Using the definition of $S_k$ yields
\begin{equation}  
    \begin{aligned} 
        P_k + \bar{P}_k^\infty=& Q + \Gamma ^ T Q \Gamma - Q \Gamma - \Gamma ^ T Q \\
        &+  A^T (P_{k+1}+\bar{P}_{k+1}^\infty)  A + \sigma^2  C^TP_{{k+1}}C  \\
        &- {{}(\bar{K}_{k}^{\infty})}^T \big[ R+B^T(P_{k+1}+\bar{P}_{k+1}^\infty) B \\
        &+ \sigma^2 D^TP_{k+1}D \big] {\bar{K}_{k}^{\infty}},
    \end{aligned}   \label{pf2_4}
\end{equation}
which can be written as
\begin{equation}    
    \begin{aligned} 
        &\bar{P}_{k}^\infty=\frac{N-1}{N}(\Gamma^TQ\Gamma-Q\Gamma-\Gamma^TQ)+A^T \bar{P}_{k+1}^\infty A \\
        &+(B^TP_{k+1}A+\sigma^2D^TP_{k+1}C)^T(R+B^TP_{k+1}B \\
        &+\sigma^2D^TP_{k+1}D)^{-1}(B^TP_{k+1}A+\sigma^2D^TP_{k+1}C)\\
        &-(B^TP_{k+1}A+\sigma^2D^TP_{k+1}C+B^T \bar{P}_{k+1}^\infty A)^T \\
        &\times (R+B^TP_{k+1}B+\sigma^2D^TP_{k+1}D+B^T \bar{P}_{k+1}^\infty B)^{-1} \\
        &\times  (B^TP_{k+1}A+\sigma^2D^TP_{k+1}C+B^T \bar{P}_{k+1}^\infty A).
    \end{aligned} 
\end{equation}
Since 
$\bar{P}_{T+1}= \bar{P}_{T+1}^\infty= \frac{N-1}{N}(\Gamma^TQ\Gamma-Q\Gamma-\Gamma^TQ)$
and the ARE \eqref{e1_6} has a unique positive semi-definite solution, we have 
$\bar{P}_{k}= \bar{P}_{k}^\infty$.
{{\hfill{$\Box$}}}

\subsection{Gain matrix approximation with unknown dynamics}

First, based on the model-based iterations \eqref{th1_1}-\eqref{th1_2}, 
we design an equivalent model-free iteration to approximate the control gains $\{K_k\}_{k=0}^T$.

For a fixed $1 \leq j \leq N$, we have
 \begin{equation} 
    \begin{aligned} 
        & \mathbb{E} \Big\{ x_j^T(k+1)P_{k}x_j(k+1)-x_j^T(k)P_{k-1}x_j(k) \Big \} \\
        =& \mathbb{E} \Big\{ (Ax_j(k)+Bu_j(k))^TP_{k}(Ax_j(k)+Bu_j(k)) \\
        &+\sigma^2(Cx_j(k)+Du_j(k))^TP_{k}(Cx_j(k)+Du_j(k)) \\
        &-x_j^T(k) P_{k-1}x_j(k) \Big\}\\
        =& \mathbb{E} \Big\{ -x_j^T(k) \big[ Q+\frac{1}{N}(\Gamma^TQ\Gamma-Q\Gamma-\Gamma^TQ)  \\
        &+K_{k-1}^TRK_{k-1} \big] x_j(k) \\
        & +2(u_j(k)-K_{k-1}x_j(k))^T \Lambda_{2}^kx_j(k) \\
        & +u_i^T(k)\Lambda_{1}^k u_j(k)-x_j^T(k)K_{k-1}^T\Lambda_{1}^kK_{k-1}x_j(k) \Big\}, 
    \end{aligned}  \label{e2_1}
\end{equation} 
where $\Lambda_{1}^k=B^TP_{k}B+\sigma^2D^TP_{k}D, \Lambda_{2}^k=-(R+\Lambda_{1}^{k+1})K_{k}$.
By vectorization, one has 
 \begin{equation} 
    \begin{aligned} 
        & \mathbb{E} \Big\{ x_j^T(k+1)P_{k}x_j(k+1)\Big\} - \mathbb{E} \Big\{ (\tilde{x}_j(k))^T vecs(P_{k-1}^i) \Big\}\\
        =& \mathbb{E} \Big\{ -x_j^T(k) \big[ Q+\frac{1}{N}(\Gamma^TQ\Gamma-Q\Gamma-\Gamma^TQ) \\
        &  +{(K_{k-1}^{i})}^TRK_{k-1}^{i} \big] x_j(k) \Big\} \\
        & +2 \mathbb{E} \Big\{ x_j^T(k) \otimes (u_j(k)-K_{k-1}^{i} x_j(k))^T \Big\} vec(\Lambda_{2}^{k,i}) \\
        &+ \mathbb{E} \Big\{ \tilde{u}_j(k)-\widetilde{K_{k-1}^{i} x_j(k)} \Big\}^T vec(\Lambda_{1}^{k,i}),
    \end{aligned}  \label{e2_2}
\end{equation}  
To approximate the unknown parameters, at least $\frac{n}{2}(n+1)+mn+\frac{m}{2}(m+1)$ data sets are collected to 
solve \eqref{e2_2} using least squares. 
The data matrices are defined as

 \begin{equation}\medmath{
    \left\{ 
        \begin{aligned}
        &\mathcal{I}_{x_{j}}^k\triangleq \Big[ I_{x_{j} x_{j}}^{k,1},I_{x_{j} x_{j}}^{k,2},\cdots,I_{x_{j} x_{j}}^{k,l} \Big]^{\mathrm{T}},I_{x_{j} x_{j}}^{k}\triangleq \mathbb{E}(\tilde{x}_{j}(k)),\\
        &\mathcal{I}_{u_{j}}^k\triangleq \Big[ I_{u_{j} u_{j}}^{k,1},I_{u_{j} u_{j}}^{k,2},\cdots,I_{u_{j} u_{j}}^{k,l} \Big]^{\mathrm{T}},I_{u_{j} u_{j}}^{k}\triangleq \mathbb{E}(\tilde{u}_{j}(k)),\\
        &\mathcal{I}_{x_{j} u_{j}K}^k\triangleq \Big[ I_{x_{j} u_{j}K}^{k,1},I_{x_{j} u_{j}K}^{k,2},\cdots,I_{x_{j} u_{j}K}^{k,l} \Big]^{\mathrm{T}},\\
        &I_{x_{j} u_{j}K}^{k}\triangleq \mathbb{E} \Big\{x_{j}^T(k) \otimes (u_{j}(k)-K_{k-1}^{i} x_{j}(k))^T \Big\}^T,\\
        &\mathcal{I}_{u_{j}x_{j} }^k\triangleq \Big[ I_{u_{j} x_{j} }^{k,1},I_{u_{j}x_{j} }^{k,2},\cdots,I_{u_{j}x_{j} }^{k,l} \Big]^{\mathrm{T}},\\
        &I_{u_{j}x_{j} }^{k}\triangleq \mathbb{E} \Big\{  \tilde{u}_{j}(k) - \widetilde{K_{k-1}^{i} x_{j}(k)} \Big\},\\
        &\mathcal{I}_{x_{j} u_{j}}^k\triangleq \Big[ I_{x_{j} u_{j}}^{k,1},I_{x_{j} u_{j}}^{k,2},\cdots,I_{x_{j} u_{j}}^{k,l} \Big]^{\mathrm{T}}, \\
        &I_{x_{j} u_{j}}^{k}\triangleq \mathbb{E} \Big\{ x_{j}^T(k) \otimes u^T_{j}(k) \Big\}^T,\\
        & \Xi_{k}\triangleq \Big[ \xi^{k,1},\xi^{k,2},\cdots,\xi^{k,l} \Big]^{\mathrm{T}},\\
        & \xi^k \triangleq \mathbb{E} \Big\{x^T_{j}(k+1)P_{k}x_{j}(k+1) +x^T_{j} (k) \Big[ Q \\
        & +\frac{1}{N}(\Gamma^TQ\Gamma-Q\Gamma-\Gamma^TQ)  +{(K_{k-1}^{i})}^TRK_{k-1}^{i} \Big] x_{j} (k) \Big\}
        \end{aligned}
        \right.} \label{e2_3}
\end{equation} 
Then, by the Kronecker product, \eqref{e2_2} implies the following matrix form of linear equations
\begin{equation}
    \Psi_{k}\begin{bmatrix} vecs(P_{k-1}^i)   \\ vec(\Lambda_2^{k,i})  \\ vecs(\Lambda_1^{k,i}) \end{bmatrix}=\Xi_{k},  \label{ls1}
\end{equation} 
where 
$\Psi_{k} \triangleq [\mathcal{I}_{x_{j}}^k, 2\mathcal{I}_{x_{j} u_{j}K}^k, \mathcal{I}_{u_{j}x_{j} }^k]$.

\begin{assum} \label{ass1}
    There exists an $l_{1} >0$, such that, for all $l \geq l_{1}$, 
 \begin{equation}
        \medmath{\text{rank}\left(\begin{bmatrix}
            \mathcal{I}_{x_{j}}^k,  \,\, \mathcal{I}_{x_{j} u_{j} }^k, \,\, \mathcal{I}_{u_{j}}^k    
        \end{bmatrix}\right)
        =\frac{n(n+1)}{2}+mn+\frac{m(m+1)}{2}.} \label{assume}
    \end{equation} 
\end{assum}

\begin{thm} \label{thm3}
    Suppose Assumption \ref{ass1} holds.  $P_k$ is the positive semi-definite solution to the ARE \eqref{e1_5}.
    The sequence $\{ P_k^i, \Lambda_1^{k+1,i}, \Lambda_2^{k+1,i}  \}_{i=1}^{\infty}$  
    generated by recursively solving the following equation
    \begin{equation}
    \begin{bmatrix} vecs(P_{k-1}^i)   \\ vec(\Lambda_2^{k,i})  \\ vecs(\Lambda_1^{k,i}) \end{bmatrix}
    = (\Psi_{k}^T \Psi_{k})^{-1} \Psi_{k}^T \Xi_{k}, \label{th3_1}
\end{equation}
satisfies $\lim\limits_{ i \to \infty } P_{k}^i = P_k $, 
where 
$\Psi_{k} \triangleq [\mathcal{I}_{x_{j}}^k, 2\mathcal{I}_{x_{j} u_{j}K}^k, \mathcal{I}_{u_{j}x_{j} }^k]$. 
\end{thm}
\emph{Proof.} Equation
\eqref{th3_1} is equivalent to \eqref{th1_1} and \eqref{th1_2}.
Then combining with Theorem \ref{thm1}, the model-free method is convergent. 
Next, we show that $\Psi_{k+1}$ has full column rank by contradiction.

We first assume that $\Psi_{k+1}H=0$ for a nonzero vector $H=[X_v^T,Y_v^T,Z_v^T]^T$, 
where 
$X_v= vecs(X) \in \mathbb{R}^ {\frac{n}{2}(n+1)}, $ 
$Y_v=vec(Y) \in \mathbb{R}^{mn},  $
$Z_v= vecs(Z) \in \mathbb{R}^{\frac{m}{2}(m+1)} $.
Then, one has
$\mathcal{I}_{x_{j}}^k vecs(X) + 2\mathcal{I}_{x_{j} u_{j}K}^k  vec(Y) + \mathcal{I}_{u_{j}x_{j} }^k vecs(Z) = 0 $.
By \eqref{e2_2}, we have 
\begin{equation}
    \mathcal{I}_{x_{j}}^k  vecs(M_1 ) + \mathcal{I}_{x_{j} u_{j}}^k  vec( M_2 ) + \mathcal{I}_{u_{j}}^k  vecs( M_3 ) = 0.
\label{pf4_1}
\end{equation}
where 
\begin{equation}
    \left\{
    \begin{aligned}
        &M_1 = X - 2 Y^T{K_{k}^i} -{(K_{k}^i)}^TZ{K_{k}^i}, \\
        &M_2 = 2Y,  \\
        &M_3 = Z.    
    \end{aligned}    
    \right. 
\end{equation}
Since Assumption \ref{ass1} holds,
\eqref{pf4_1} has a unique solution. Therefore, we have $M_1 = 0, M_2 = 0, M_3 = 0$,
i.e., $X = 0, Y = 0, Z = 0$,
leading to a contradiction with the assumption that $H \neq 0$. 
This completes the proof.
\rightline{{\hfill{$\Box$}}}    

Next, we utilize the system trajectories to iterate \eqref{th2_1}-\eqref{th2_2} without the system parameters.
The dynamics of the $i$-th agent is rewritten as follows
\begin{equation}
\begin{aligned} 
\bar{x}_j(k+1) &=A\bar{x}_j(k)+B\bar{u}_j(k) \\
    & =\bar{A}_{k}\bar{x}_j(k)-B\bar{K}_{k}\bar{x}_j(k)+B\bar{u}_j(k),
\end{aligned}     \label{e2_4}
\end{equation}
and  we can deduce that
\begin{equation} 
    \begin{aligned}
        & \bar{x}_{j}^T(k+1)\bar{P}_{k+1}\bar{x}_{j}(k+1)-\bar{x}_{j}^T(k)\bar{P}_{k}\bar{x}_{j}(k) \\
        =& [\bar{A}_{k}\bar{x}_{j}(k)-B\bar{K}_{k}\bar{x}_{j}(k)+B\bar{u}_{j}(k)]^T \bar{P}_{k+1} [\bar{A}_{k}\bar{x}_{j}(k) \\
        & -B\bar{K}_{k}\bar{x}_{j}(k)+B\bar{u}_{j}(k)] - \bar{x}_{j}^T(k) \bar{P}_{k} \bar{x}_{j}(k) \\ 
        =& \bar{x}_{j}^T(k) \bar{A}_{k}^T \bar{P}_{k+1} \bar{A}_{k}\bar{x}_{j}(k) - \bar{x}_{j}^T(k) \bar{P}_{k} \bar{x}_{j}(k)  \\&
        + 2 \bar{x}_{j}^T(k) \bar{A}_{k}^T \bar{P}_{k+1} B(-\bar{K}_{k}\bar{x}_{j}(k)+\bar{u}_{j}(k))  \\&
        + \! (-\bar{K}_{k}\bar{x}_{j}(k) \!+\! \bar{u}_{j}(k))^T \! B^T \! \bar{P}_{k+1} B(-\bar{K}_{k}\bar{x}_{j}(k)+\bar{u}_{j}(k)) \\
        =& - \bar{x}_{j}^T(k)[\frac{N-1}{N}(\Gamma^TQ\Gamma-Q\Gamma-\Gamma^TQ) \\
        &+ {{}({\Lambda}_{2}^{k+1})}^{T} (R + {\Lambda}_{1}^{k+1})^{-1} {\Lambda}_{2}^{k+1} ]\bar{x}_{j}(k) \\
        &+ 2 \bar{u}_{j}^T(k) \bar{\Lambda}_{2}^{k+1} \bar{x}_{j}(k) +  \bar{u}^T_{i}(k)   \bar{\Lambda}_{1}^{k+1} \bar{u}_{j}(k) \\
        &+ \bar{x}_{j}^T(k) \bar{K}_{k}^T (R + {\Lambda}_{1}^{k+1} + \bar{\Lambda}_{1}^{k+1}) \bar{K}_{k} \bar{x}_{j}(k),
    \end{aligned}     \label{e2_5}
\end{equation} 
where 
$\bar{\Lambda}_{1}^{k} \triangleq    {B}^T \bar{P}_{k} {B} $,
$\bar{\Lambda}_{2}^{k} \triangleq    {B}^T \bar{P}_{k} {A} $,
$\bar{K}_{k}=-(R + {\Lambda}_{1}^{k+1} + \bar{\Lambda}_{1}^{k+1})^{-1} ( {\Lambda}_{2}^{k+1} + \bar{\Lambda}_{2}^{k+1})$.
Based on above equation, we can obtain the following recurrence relation
{ \begin{equation}
    \begin{aligned}
        & \bar{x}_{j}^T(k+1)\bar{P}_{k+1}\bar{x}_{j}(k+1)-\bar{x}_{j}^T(k) \bar{P}_{k}^i \bar{x}_{j}(k) \\
        =& - \bar{x}_{j}^T(k)[\frac{N-1}{N}(\Gamma^TQ\Gamma-Q\Gamma-\Gamma^TQ) \\
        &+ {{}({\Lambda}_{2}^{k+1})}^{T} (R + {\Lambda}_{1}^{k+1})^{-1} {\Lambda}_{2}^{k+1} \\
        &- (\bar{K}_{k}^{i})^T (R + {\Lambda}_{1}^{k+1} ) \bar{K}_{k}^{i} ]\bar{x}_{j}(k) \\
        &+ 2  \bar{u}_{j}^T(k) \bar{\Lambda}_{2}^{k+1,i} \bar{x}_{j}(k) + \bar{u}_{j}^T(k)  \bar{\Lambda}_{1}^{k+1,i} \bar{u}_{j}(k) \\
        &+ \bar{x}_{j}^T(k) (\bar{K}_{k}^{i})^T   \bar{\Lambda}_{1}^{k+1,i} \bar{K}_{k}^{i} \bar{x}_{j}(k).
    \end{aligned}   \nonumber
\end{equation}}
By vectorization, the above equation  becomes
 \begin{equation}  
\begin{aligned}
    & \bar{x}_{j}^T(k+1) \bar{P}_{k+1} \bar{x}_{j}(k+1)-\tilde{\bar{x}}_{j}^T(k) vecs(\bar{P}_{k}^i) \\
    =& - \bar{x}_{j}^T(k) \big[\frac{N-1}{N}(\Gamma^TQ\Gamma-Q\Gamma-\Gamma^TQ) \\
        &+ {{}({\Lambda}_{2}^{k+1})}^{T} (R + {\Lambda}_{1}^{k+1})^{-1} {\Lambda}_{2}^{k+1} \\
        &- (\bar{K}_{k}^{i})^T (R + {\Lambda}_{1}^{k+1} ) \bar{K}_{k}^{i} \big]\bar{x}_{j}(k)  \\
        &+ (\tilde{\bar u}_j(k)+\widetilde{{\bar K}_{k}^{i} {\bar x}_j(k)})  ^T vecs(\bar{\Lambda}_{1}^{k+1,i}) \\
        &+ 2 ({\bar x}_{j}^T(k) \otimes {\bar u}^T_{j}(k)) vec(\bar{\Lambda}_{2}^{k+1,i}).
\end{aligned}  \label{e2_7}
\end{equation}

Using least squares, the above equation can be solved as follows, without any knowledge of the system dynamics
\begin{equation}
    \bar{\Psi}_{k}\begin{bmatrix} vecs(\bar{P}_{k}^i) \\ vecs(\bar{\Lambda}_{1}^{k+1,i})   \\   vec(\bar{\Lambda}_{2}^{k+1,i}) \end{bmatrix}=\bar{\Xi}_{k}, \label{ls2}
\end{equation}
where 
$\bar{\Psi}_{k} \triangleq [\bar{\mathcal{I}}_{\bar{x}}^k, \bar{\mathcal{I}}_{ \bar{K} \bar{x} \bar{u} }^k,   2 \bar{\mathcal{I}}_{\bar{x} \bar{u}}^k]$.
At least $\frac{n}{2}(n+1)+ mn+\frac{m}{2}(m+1)$ data sets are collected to solve above equation.
The data matrices are defined as follows
 \begin{equation}\medmath{
    \left\{ 
        \begin{aligned}
        & \bar{\mathcal{I}}_{\bar{x}}^k\triangleq \Big[ \bar{I}_{\bar{x}}^{k,1}, \bar{I}_{\bar{x}}^{k,2},\cdots, \bar{I}_{\bar{x}}^{k,l} \Big] ^{\mathrm{T}}, 
        \bar{I}_{\bar{x}}^{k}\triangleq  \tilde{\bar{x}}_j(k) ,\\        
        & \bar{\mathcal{I}}_{\bar{u}}^k\triangleq \Big[ \bar{I}_{\bar{u}}^{k,1}, \bar{I}_{\bar{u}}^{k,2},\cdots, \bar{I}_{\bar{u}}^{k,l} \Big] ^{\mathrm{T}}, 
        \bar{I}_{\bar{u}}^{k}\triangleq  \tilde{\bar{u}}_j(k) ,\\ 
        & \bar{\mathcal{I}}_{ \bar{K} \bar{x} \bar{u}}^k\triangleq \Big[ \bar{I}_{ \bar{K} \bar{x}\bar{u}}^{k,1}, \bar{I}_{ \bar{K} \bar{x}\bar{u}}^{k,2},\cdots, \bar{I}_{ \bar{K} \bar{x}\bar{u}}^{k,l} \Big]^{\mathrm{T}}, \\
        & I_{\bar{K} \bar{x}\bar{u}}^{k}\triangleq  \widetilde{   \bar{K}_{k}^{i} \bar{x}_j (k)} + \tilde{\bar{u}}_j(k)   ,\\        
        & \bar{\mathcal{I}}_{\bar{x} \bar{u}}^k\triangleq \Big[ \bar{I}_{\bar{x} \bar{u}}^{k,1}, \bar{I}_{\bar{x} \bar{u}}^{k,2},\cdots, \bar{I}_{\bar{x} \bar{u}}^{k,l} \Big]^{\mathrm{T}}, \\
        & \bar{I}_{\bar{x} \bar{u}}^{k}\triangleq  \Big\{ \bar{x}^T_j(k) \otimes  \bar{u}^T_j(k)  \Big\}^T,  \\
        & \bar{\Xi}_{k}\triangleq \Big[\bar{\xi}^{k,1},\bar{\xi}^{k,2},\cdots,\bar{\xi}^{k,l} \Big]^{\mathrm{T}},\\
        & \bar{\xi}^k \triangleq \bar{x}_j^T(k+1) \bar{P}_{k+1} \bar{x}_j(k+1) \\
        & + \bar{x}_j^T(k) \Big[ \frac{N-1}{N}(\Gamma^TQ\Gamma-Q\Gamma-\Gamma^TQ) \\
        & + {{}({\Lambda}_{2}^{k+1})}^{T} (R + {\Lambda}_{1}^{k+1})^{-1} {\Lambda}_{2}^{k+1} \\
        & - (\bar{K}_{k}^{i})^T (R + {\Lambda}_{1}^{k+1} ) \bar{K}_{k}^{i}  \Big]\bar{x}_j(k) .
        \end{aligned}
        \right.}  \nonumber
\end{equation}  
\begin{assum} \label{ass2}
There exists an $l_2 > 0$, such that, for all $l \geq l_2$, 
 \begin{equation}
\text{rank}\left(\begin{bmatrix}
    \bar{\mathcal{I}}_{\bar{x}}^k, \,\, \bar{\mathcal{I}}_{\bar{x} \bar{u}}^k, \,\, \bar{\mathcal{I}}_{ \bar{u}}^k   
\end{bmatrix}\right)
=\frac{n}{2}(n+1)+ mn+\frac{m}{2}(m+1). \nonumber
\end{equation} 
\end{assum}   

\begin{thm} \label{thm4}
	Suppose Assumption \ref{ass2} holds. $\bar{P}_{k}$ is the positive semi-definite solution to the ARE \eqref{e1_6}. 
    The sequence $\{ \bar{P}_{k}^i, \bar{\Lambda}_{1}^{k+1,i}, \bar{\Lambda}_{2}^{k+1,i} \}_{i=0}^{\infty}$  generated by recursively solving the following equation
$$
     \begin{bmatrix} vecs(\bar{P}_{k}^i) \\ vecs(\bar{\Lambda}_{1}^{k+1,i})   \\   vec(\bar{\Lambda}_{2}^{k+1,i}) \end{bmatrix}
    =(\bar{\Psi}_{k}^T \bar{\Psi}_{k})^{-1} \bar{\Psi}_{k}^T \bar{\Xi}_{k},
$$
satisfies $\lim\limits_{ i \to \infty } \bar{P}_{k}^i = \bar{P}_{k} $. 
\end{thm}

\emph{Proof.}
We first assume that $\bar{\Psi}_{k} \bar{H} = 0$ for a nonzero vector $\bar{H} = [\bar{H}_{v,1}^T,\bar{H}_{v,2}^T,\bar{H}_{v,3}^T]^T$,
where $\bar{H}_{v,1}=vecs(H_1),$ $\bar{H}_{v,2}=vecs(H_2),\bar{H}_{v,3}=vec(H_3)$.
Then, one has 
$ \bar{\mathcal{I}}_{\bar{x}}^k \bar{H}_{v,1} + \bar{\mathcal{I}}_{ \bar{K} \bar{x} \bar{u} }^k \bar{H}_{v,2} + 2 \bar{\mathcal{I}}_{\bar{x} \bar{u}}^k \bar{H}_{v,3} =0$.
By \eqref{e2_7}, we have  
\begin{equation}
   \bar{ \mathcal{I}}_{\bar{x} }^k  vecs(\bar{M_1}) + \bar{\mathcal{I}}_{\bar{x} \bar{u} }^k  vec(\bar{M_2}) + \bar{\mathcal{I}}_{\bar{u} }^k  vecs(\bar{M_3})   = 0,
\label{pfls2_1}
\end{equation}
where 
\begin{equation}
    \left\{
    \begin{aligned}
        \bar{M}_1 &= H_1 + {(\bar{K}_{k}^{i})}^T H_2 { \bar{K}_{k}^{i} }  , \\
        \bar{M}_2 &= 2H_3 + 2 H_2 { \bar{K}_{k}^{i} }   ,  \\
        \bar{M}_3 &= H_2 .    
    \end{aligned}   \nonumber
    \right. 
\end{equation}
Since Assumption \ref{ass2} holds,
\eqref{pfls2_1} has a unique solution. Therefore, we have $\bar{M}_1 = 0, \bar{M}_2 = 0, \bar{M}_3 = 0$,
i.e., $H_1 = 0, H_2 = 0, H_3 = 0$,
leading to a contradiction with the assumption that $\bar{H} \neq 0$. 
This completes the proof. 
{{\hfill{$\Box$}}}   

\begin{rem}
    To satisfy the persistence of excitation condition, probing noise is added to the system.
    If the collected data do not satisfy the rank condition, 
    probing noise is injected into the control input and additional data are collected until the condition is satisfied.
    In other words, emoploy $u_j = K_0 x_j + e_j$, where $e_j$ is the probing noise. 
    The selection of probing noise is no simple matter in general RL \cite{Jiang2012Computational}.
    There are several types of exploration noise \cite{XU2025Mean}, such as random noise, exponentially decreasing probing noise, sum of sinusoidal signals with different frequencies.
\end{rem}

\begin{rem}
    The proposed two off-policy RL algorithms iteratively solve \eqref{ls1} and \eqref{ls2}, respectively.
    The algorithms do not require the knowledge of the system dynamics and 
    their behavior policy is unrelated to the extimated policy,  which enable us to make use of the previously produced data for learning.
\end{rem} 

\section{Main results for Problem 2}

In this section, we investigate properties of the proposed model-free algorithm as $T \to \infty$ (infinite-horizon
case).

\begin{assum} \label{assum_stab}
        System \eqref{e1_1}  $[A,B;C,D]$  is $L^2-stabilizable$ and
        $\bar{\mathcal{M}} \neq \varnothing$, where   $\bar{\mathcal{M}} \triangleq \{(P,S) =(P^T,S^T) | \bar{\mathcal{H}}(P,S) > 0, \mathcal{H}(P) > 0\}  $,
\\         
$
\left\{
\begin{aligned}
\overline{\mathcal{H}}(P,S) &\triangleq 
\begin{bmatrix}
\bar{L}(P,S) & \bar{H}^T(P,S) \\
\bar{H}(P,S) & \bar{W}(P,S)
\end{bmatrix}, \\
\mathcal{H}(P) &\triangleq 
\begin{bmatrix}
L(P) & H^T(P) \\
H(P) & W(P)
\end{bmatrix},
\end{aligned}
\right.
$\\
with\\
$
\left\{
\begin{aligned}
\bar{L}(P,S) &\triangleq  \frac{N-1}{N}(\Gamma ^ T Q \Gamma - Q \Gamma - \Gamma ^ T Q) \\
&+ \sigma^2 C ^T P C    
 + A ^T S A - S, \\
\bar{H}(P,S) &\triangleq  B ^T S  A  
 + \sigma^2 D ^T P C,   \\
\bar{W}(P,S) &\triangleq R + B ^T S B + \sigma^2 D  ^T P D   ,  \\
L(P) &\triangleq Q + \frac{1}{N}(\Gamma ^ T Q \Gamma - Q \Gamma - \Gamma ^ T Q) \\
&+ A^T P A + \sigma^2 C^T P C - P, \\
H(P) &\triangleq B^T P A + \sigma^2 D^T P C,  \\
W(P) &\triangleq R + B^T P B + \sigma^2 D^T P D.
\end{aligned}
\right.
$
\end{assum}

\subsection{Model-based policy iteration}

\begin{lem} \cite[Proposition 3.2]{ni2015indefinite}
    Under Assumption \ref{assum_stab}, 
    there exists a stabilizing solution 
    to the following coupled AREs \eqref{section4_1}-\eqref{section4_2}:
    \begin{equation}
        \begin{aligned}
            &P=Q+\frac{1}{N}(\Gamma^TQ\Gamma-Q\Gamma-\Gamma^TQ)+A^TPA \\
            &+\sigma^2C^TPC  -(B^TPA+\sigma^2D^TPC)^T \\
            &\times   (R+B^TPB+\sigma^2D^TPD)^{-1} \\
            &\times  (B^TPA+\sigma^2D^TPC), 
        \end{aligned}     \label{section4_1}
    \end{equation}
    \begin{equation}      
        \begin{aligned} 
            &\bar{P}=\frac{N-1}{N}(\Gamma^TQ\Gamma-Q\Gamma-\Gamma^TQ)+A^T\bar{P}A \\
            &+(B^TPA+\sigma^2D^TPC)^T(R+B^TPB \\
            &+\sigma^2D^TPD)^{-1}(B^TPA+\sigma^2D^TPC)\\
            &-(B^TPA+\sigma^2D^TPC+B^T\bar{P}A)^T \\
            &\times (R+B^TPB+\sigma^2D^TPD+B^T\bar{P}B)^{-1} \\
            &\times  (B^TPA+\sigma^2D^TPC+B^T\bar{P}A). 
        \end{aligned} \label{section4_2}
    \end{equation}
\end{lem}

\begin{thm} \label{thm5}
Let $K^0$ be any stabilizing gain matrix, and let $P^{i}$ be the solution of the Lyapunov equation \eqref{th4_1}.
\begin{equation}
    \begin{aligned} 
        P^i=& Q+\frac{1}{N}(\Gamma^TQ\Gamma-Q\Gamma-\Gamma^TQ)
        +{(A^i)}^TP^iA^i \\
        & +\sigma^2{(C^i)}^TP^iC^i+{(K^i)}^TRK^i, 
    \end{aligned} \label{th4_1}
\end{equation}
where $K^{i}$, with $i=1,2,...,$ are defined recursively by: 
\begin{equation}
    \begin{aligned} 
        K^{i+1}=& -(R+B^TP^iB+\sigma^2 D^TP^iD)^{-1} \\
        &\times  (B^TP^iA+\sigma^2 D^TP^iC),
    \end{aligned} \label{th4_2}
\end{equation}
where 
$A^i \triangleq A+BK^i$ and $C^i \triangleq C+DK^i$. 
Then, under Assumption \ref{assum_stab},  the following properties hold:\\
(1). $K^{i}$ is a stabilizer of the system $[A,B;C,D|Q]$.\\
(2).  $P \leq P^{i+1} \leq P^i$. \\
(3). $\lim\limits_{i\to+\infty} P^i = P$, $\lim\limits_{i\to+\infty} K^i = K$.
\end{thm}
\emph{Proof.}
We define the Lyapunov-type operator as follows
 	\begin{equation}\medmath{
		   \mathcal{L}_i(X) \triangleq (A + B K^{i})^T X (A + B K^{i})  + \sigma^2 (C + D K^{i})^T X (C + D K^{i}), \nonumber
	}\end{equation}	 
	with spectrum  
	\begin{equation}
        \sigma(\mathcal{L}_i(X)) \triangleq \{ \lambda \in \mathbb{R} | \mathcal{L}_i(X) = \lambda X, X \in \mathbb{S}^{n}, X \neq 0  \}. \nonumber
    \end{equation}
    If $i=0$, the initial conditon $K^0$ ensures result (1) holds.
    Suppose the result (1) holds for $i-1 > 0$, we prove that the result (1) holds for $i$.
    Assume there exists $\lambda \in \sigma(\mathcal{L}_i(X))$ such that $\lambda \geq 1$.
    Subtracting \eqref{section4_1} from \eqref{th4_1} yields 
\begin{equation}
        \begin{aligned} 
            & P^i - P =  
            {(A^i)}^T(P^i - P)A^i 
             +\sigma^2{(C^i)}^T(P^i - P)C^i\\
             &+{(K^i - K)}^T( R + B^TPB + \sigma^2 D^TPD )(K^i - K). 
        \end{aligned} \nonumber
\end{equation}
This can be written as
 \begin{equation}\medmath{
    \begin{aligned} 
         & {(A^{i-1})}^T(P^{i-1} - P)A^{i-1} 
         +\sigma^2{(C^{i-1})}^T(P^{i-1} - P)C^{i-1} \\
         &+{(K^{i-1} - K)}^T( R + B^TPB + \sigma^2 D^TPD )(K^{i-1} - K)\\
         &  - (P^{i-1} - P)= 0. 
    \end{aligned}} \label{section4_4}
\end{equation} 
Since $K^{i-1}$ is a stabilizer and $R + B^TPB + \sigma^2 D^TPD > 0$, one has $P^{i-1} - P \geq 0$.
The above equation can be transformed into
  \begin{equation}\medmath{
    \begin{aligned} 
         & {(A^{i})}^T(P^{i-1} - P)A^{i} 
         +\sigma^2{(C^{i})}^T(P^{i-1} - P)C^{i} 
           - (P^{i-1} - P) \\
         &= -{(K^{i} - K)}^T( R + B^TPB + \sigma^2 D^TPD )(K^{i} - K)\\
         &-{(K^{i-1} - K^{i})}^T( R + B^TP^{i-1}B + \sigma^2 D^TP^{i-1}D )(K^{i-1} - K^{i}).
    \end{aligned}}  \nonumber
\end{equation} 
Then, we have
\begin{equation}
    \begin{aligned} 
        &(\mathcal{A}_i - 1 ) vec(P^{i-1} - P)  \\
        =& - [{(K^{i} - K)}^T \otimes {(K^{i} - K)}^T] \\
        &\times vec(R + B^TPB + \sigma^2 D^TPD)\\
        &- [{(K^{i-1} - K^{i})}^T \otimes {(K^{i-1} - K^{i})}^T] \\
        & \times vec(R + B^TP^{i-1}B + \sigma^2 D^TP^{i-1}D),
    \end{aligned} \nonumber
\end{equation}
where $\mathcal{A}_i= (A^i)^T \otimes (A^i)^T + \sigma^2 (C^i)^T \otimes (C^i)^T$.
Premultiplying the above equation by $vec^T(X)$, we have
\begin{equation}
    \begin{aligned} 
        & vec^T(X)(\lambda - 1 ) vec(P^{i-1} - P)   \\
       = &- vec^T(X)[{(K^{i} - K)}^T \otimes {(K^{i} - K)}^T] \\
        & \times vec(R + B^TPB + \sigma^2 D^TPD)\\
        &- vec^T(X)[{(K^{i-1} - K^{i})}^T \otimes {(K^{i-1} - K^{i})}^T] \\
        & \times vec(R + B^TP^{i-1}B + \sigma^2 D^TP^{i-1}D),
    \end{aligned} \nonumber
\end{equation}
or equivalently,
\begin{equation}
    \begin{aligned} 
        & (\lambda - 1 ) Tr [(P^{i-1} - P)X]   \\
       = &- Tr  [X^{1/2}{(K^{i} - K)}^T (R + B^TPB + \sigma^2 D^TPD) \\
        &\times {(K^{i} - K)}X^{1/2} 
        + X^{1/2}{(K^{i-1} - K^{i})}^T \times (R \\
        & + B^TP^{i-1}B + \sigma^2 D^TP^{i-1}D) {(K^{i-1} - K^{i})} X^{1/2} ].
    \end{aligned} \label{section4_5}
\end{equation}
As $P^{i-1} - P \geq 0$ and $-{(K^{i} - K)}^T (R + B^TPB + \sigma^2 D^TPD){(K^{i} - K)} - {(K^{i-1} - K^{i})}^T(R + B^TP^{i-1}B + \sigma^2 D^TP^{i-1}D) {(K^{i-1} - K^{i})} \leq 0$,
the equation \eqref{section4_5} implies ${(K^{i-1} - K^{i})}X=0$. 
Consequently, there exists $\lambda \in \sigma(\mathcal{L}_{i-1}(X))$ such that $\lambda \geq 1$.
It contradicts with the induction assumption. Therefore, the result (1) holds for $i \in \mathbb{N}^+ $.

Replacing $i$ in \eqref{th4_1} by $i-1$, one has
\begin{equation}
    \begin{aligned} 
        &P^{i-1}= Q+\frac{1}{N}(\Gamma^TQ\Gamma-Q\Gamma-\Gamma^TQ) \\
        & +{(A^{i-1})}^TP^{i-1}A^{i-1} +\sigma^2{(C^{i-1})}^TP^{i-1}C^{i-1} \\
        & +{(K^{i-1})}^TRK^{i-1}. 
    \end{aligned} \label{section4_6}
\end{equation}
Subtracting  \eqref{th4_1} from \eqref{section4_6}, one has
\begin{equation}
    \begin{aligned} 
        P^{i-1}-P^{i}=&{(A^{i})}^T(P^{i-1}-P^{i})A^{i} \\
        &+\sigma^2{(C^{i})}^T(P^{i-1}-P^{i})C^{i}\\ 
        &+(K^{i-1}-K^{i})^T (R+B^TP^{i-1}B \\
        & +\sigma^2 D^TP^{i-1}D)(K^{i-1}-K^{i}).  
    \end{aligned} \label{section4_7}
\end{equation}
Since $K^{i}$ is a stabilizer and $R + B^TP^{i-1}B + \sigma^2 D^TP^{i-1}D > 0$, 
one has $P^{i-1} - P^{i} \geq 0$. From \eqref{section4_4} and \eqref{section4_7}, the result (2) holds for $i \in \mathbb{N}^+$.

Therefore, the sequence  $\{ P^i \}_{i=0}^{\infty}$ is monotonically decreasing and has a lower bound. 
This sequence has a limit.
The sequence  $\{ K^i \}_{i=0}^{\infty}$ has a limit since $K^i$ is the unique solution of \eqref{th4_2} and the convergence of $\{ P^i \}_{i=0}^{\infty}$ is guaranteed.
Besides, $P$ satisfies \eqref{th4_1} with $K^i = K$. 
We have $\lim\limits_{i\to+\infty} P^i = P$ and $\lim\limits_{i\to+\infty} K^i = K$. 
This completes the proof.  
{{\hfill{$\Box$}}}

\begin{rem}
    Compared to the finite-horizon case (Theorem \ref{thm1}), 
    Theorem \ref{thm5} requires an additional assumption that $K^0$ is a stabilizer, 
    and we needs to show that $K^i$ is a stabilizer for all $i \in \mathbb{N}^+$.
\end{rem}

Combining \eqref{section4_1} with \eqref{section4_2}, one has
\begin{equation} 
\begin{aligned} 
S =& Q + \Gamma ^ T Q \Gamma - Q \Gamma - \Gamma ^ T Q + \bar{A}^T S \bar{A} \\
& + \sigma^2 \bar{C}^T P \bar{C} + \bar{K}^T R \bar{K},
\end{aligned} \label{S_1}  
\end{equation}
where 
$S \triangleq P + \bar{P}$, $\bar{A} \triangleq A+B \bar{K}$ and $\bar{C} \triangleq C+D \bar{K}$.

Similar to Theorem \ref{thm5}, we obtain the following result.

\begin{thm} \label{thm6}
    Let $\bar{K}^0$ be any stabilizing gain matrix, and let $S^i$ be the solution of the Lyapunov equation \eqref{section4_8}.
\begin{equation}   
\begin{aligned}
    S^i =& Q + \Gamma ^ T Q \Gamma - Q \Gamma - \Gamma ^ T Q + {{}(\bar{A}^i)}^{T} S^i {\bar{A}^i} \\
    &+ \sigma^2 {{}(\bar{C}^i)}^{T} P \bar{C}^i + {{}(\bar{K}^i)}^{T} R \bar{K}^i,
\end{aligned}   \label{section4_8}
\end{equation}
where $\bar{K}^{i}$, with $i = 1,2,...,$ are defined recursively by: 
\begin{equation}   
\begin{aligned}
    \bar{K}^{i+1} + R^{-1} (B^TS^i  {\bar{A}^{i+1}} + \sigma^2 D^T P {\bar{C}^{i+1}}) =0.
\end{aligned}   \label{section4_9}
\end{equation}
Then, under Assumption \ref{assum_stab}, the following properties hold:\\
(1). $\bar{K}^{i}$ is a stabilizer of the system $[A,B;C,D|Q]$.\\
(2).  $S \leq S^{i+1} \leq S^i$, \\
(3). $\lim\limits_{i\to+\infty} S^i = S$, $\lim\limits_{i\to+\infty} \bar{P}^i = \bar{P}$, $\lim\limits_{i\to+\infty} \bar{K}^i = \bar{K}$.
\end{thm}

\emph{Proof.}
The proof of Theorem \ref{thm6} is similar to that of Theorem \ref{thm5},
    using the assumed initial stabilizer $\bar{K}^0$ and mathematical induction.    
The Lyapunov-type operator is defined as 
$\medmath{
\bar{\mathcal{L}}_i(X) \triangleq (A + B \bar{K}^{i})^T X (A + B \bar{K}^{i}) .
	}
$    
Similarly, combining \eqref{S_1} with \eqref{section4_8}, one has
 \begin{equation}\medmath{
    \begin{aligned} 
         & {(\bar{A}^{i-1})}^T(S^{i-1} - S)\bar{A}^{i-1} - (S^{i-1} - S)  \\
         = &-{(\bar{K}  - \bar{K}^{i-1} )}^T( R + B^TSB + \sigma^2 D^TPD ) (\bar{K}  - \bar{K}^{i-1} ),
    \end{aligned}} \label{section_proof_s1}
\end{equation} 
or equivalently,
\begin{equation}
    \begin{aligned} 
        &( (\bar{A}^i)^T \otimes (\bar{A}^i)^T - 1 ) vec(S^{i-1} - S) \\
       =  & - [{(\bar{K} - \bar{K}^{i})}^T \otimes {(\bar{K} - \bar{K}^{i})}^T] \\
        &\times vec(R + B^TSB + \sigma^2 D^TPD)\\
        &- [{(\bar{K}^{i-1} - \bar{K}^{i})}^T \otimes {(\bar{K}^{i-1} - \bar{K}^{i})}^T] \\
        & \times vec(R + B^TS^{i-1}B + \sigma^2 D^TP D).
    \end{aligned} \label{section_proof_s2}
\end{equation}
If $i=0$, the initial conditon $\bar{K}^0$ ensures result (1) holds.
Suppose the result (1) holds for $i-1 > 0$, i.e., $\bar{K}^{i - 1}$ is a stabilizer. Then it follows from \eqref{section_proof_s1} that $S^{i-1} - S \geq 0$.
Assume there exists $\bar{\lambda} \in \sigma(\bar{\mathcal{L}}_i(X))$ such that $\bar{\lambda} \geq 1$,
where $\sigma(\bar{\mathcal{L}}_i(X))$ denotes the spectrum of the Lyapunov-type operator $\bar{\mathcal{L}}_i(X)$.
From \eqref{section_proof_s2}, one has
\begin{equation}
    \begin{aligned} 
        & (\bar{\lambda} - 1 ) Tr [(S^{i-1} - S)X] \\
        = & - Tr  [X^{1/2}{(\bar{K} - \bar{K}^{i})}^T (R + B^TSB + \sigma^2 D^TPD) \\
        &\times {(\bar{K} - \bar{K}^{i})}X^{1/2} 
        + X^{1/2}{(\bar{K}^{i-1} - \bar{K}^{i})}^T \times (R \\
        & + B^TS^{i-1}B + \sigma^2 D^TP D) {(\bar{K}^{i-1} - \bar{K}^{i})} X^{1/2} ].
    \end{aligned} \nonumber
\end{equation}
Above equation  implies ${(\bar{K}^{i-1} - \bar{K}^{i})}X=0$. 
Consequently, there exists $\lambda \in \sigma(\mathcal{L}_{i-1}(X))$ such that $\lambda \geq 1$.
It contradicts with the induction assumption. Therefore, the result (1) holds for $i \in \mathbb{N}^+ $.
From \eqref{section4_8}, one has
\begin{equation}
    \begin{aligned} 
        S^{i-1}-S^{i}=&{(\bar{A}^{i})}^T (S^{i-1}-S^{i}) \bar{A}^{i} \\ 
        &+(\bar{K}^{i-1}-\bar{K}^{i})^T (R+B^T S^{i-1} B \\
        & +\sigma^2 D^TPD)(\bar{K}^{i-1}-\bar{K}^{i})  . 
    \end{aligned} \label{section_proof_s4}
\end{equation}
Since $\bar{K}^{i}$ is a stabilizer and $R + B^T S^{i-1}B + \sigma^2 D^T P D > 0$, 
one has $S^{i-1}-S^{i} \geq 0$. The result (2) holds for $i \in \mathbb{N}^+$.

Therefore, the sequence  $\{ S^i \}_{i=0}^{\infty}$ is monotonically decreasing and has a lower bound. 
This sequence has a limit.
Since $\bar{K}^{i}$ is the unique solution of \eqref{section4_9},
the convergence of $\{ \bar{K}^i \}_{i=0}^{\infty}$  can be deduced from the convergence of the sequence $\{ S^i \}_{i=0}^{\infty}$. 
Besides,
$S$ satisfies \eqref{section4_8} with $\bar{K}^i = \bar{K}$. 
This implies that
$\lim\limits_{i\to+\infty} S^i = S$
and $\lim\limits_{i\to+\infty} \bar{K}^i = \bar{K}$.
By Theorem \ref{thm5}  and ARE \eqref{section4_2} has a unique solution, one has
$\lim\limits_{i\to+\infty} \bar{P}^i = \bar{P}$.  
This completes the proof. 
\rightline{{\hfill{$\Box$}}}

\begin{rem}
To address the selection issue of initial gains of \eqref{th4_1}-\eqref{th4_2} and \eqref{section4_8}-\eqref{section4_9}, 
which are necessary to start the policy iteartion in the infinite-horizon case,
we use the data-driven approach based on system identification.
According to \cite[Theorem 6]{cui2024robust} or \cite[Theorem 1]{rami2000linear},
when the number of samples is large, 
we can compute the initial stabilizing gains  $\hat{K}^0 = -\mathcal{Y} \mathcal{X}^{-1}$ and $\hat{\bar{K}}^0 = -\bar{\mathcal{Y}} \bar{\mathcal{X}}^{-1}$ 
by solving the following linear matrix inequalities (LMIs)
\begin{equation} \medmath{
    \begin{bmatrix}
    -\mathcal{X}  & \hat{A}\mathcal{X}+\hat{B}\mathcal{Y} & C\mathcal{X}+D\mathcal{Y} \\
    \mathcal{X}\hat{A}^T+\mathcal{Y}^T\hat{B}^T & \mathcal{X}               & 0              \\
    \mathcal{X}C^T+\mathcal{Y}^T D^T & 0                         & \mathcal{X} 
    \end{bmatrix} < 0,} \nonumber
\end{equation} 
\begin{equation} \medmath{
    \begin{bmatrix}
    -\bar{\mathcal{X}}  & \hat{A}\bar{\mathcal{X}}+\hat{B}\bar{\mathcal{Y}} & C\bar{\mathcal{X}}+D\bar{\mathcal{Y}} \\
    \bar{\mathcal{X}}\hat{A}^T+\bar{\mathcal{Y}}^T\hat{B}^T & \bar{\mathcal{X}}             & 0              \\
    \bar{\mathcal{X}}C^T+\bar{\mathcal{Y}}^T D^T & 0                         & \bar{\mathcal{X}} 
    \end{bmatrix} < 0.} \nonumber
\end{equation} 
\end{rem}

\begin{rem}
The initial stabilizing gains  of algorithms are not unique.
The above inequalities may have feasible solutions in different directions, which
may give rise to different initial stabilizing gains \cite{rami2000linear}.
LMIs are powerful tools for analyzing and solving control problems, 
particularly those related to Riccati equations, stability, well-posedness, and robustness.  
The feasibility of an LMI condition is shown to be equivalent to the solvability of 
the generalized difference Riccati equation, the well-posedness, and the attainability of the LQG problem \cite{rami2002discrete}.
\end{rem} 

\subsection{Gain matrix approximation with unknown dynamics}

Similar to Section 3.2, we obtain the following design of model-free algorithm and convergence analysis.
By Kronecker product representation, we have
\begin{equation}\medmath{
    \begin{aligned} 
        & \mathbb{E} \Big\{\tilde{x}^T_j(k+1) - \tilde{x}_j^T(k)  \Big\} vecs(P^i) \\
        =& \mathbb{E} \Big\{ -x_j^T(k) \big[ Q+\frac{1}{N}(\Gamma^TQ\Gamma-Q\Gamma-\Gamma^TQ)  
         +{(K^{i})}^TRK^{i} \big] x_j(k) \Big\} \\
        & +2 \mathbb{E} \Big\{ x_j^T(k) \otimes (u_j(k)-K^{i} x_j(k))^T \Big\} vec(\Lambda_{2}^{j}) \\
        &+ \mathbb{E} \Big\{ \tilde{u}_j(k)-\widetilde{K^{i} x_j(k)} \Big\}^T vec(\Lambda_{1}^{j}),
    \end{aligned}} \nonumber
\end{equation}
and 

\begin{equation}\medmath{
\begin{aligned}
    & [ \tilde{\bar{x}}_{j}^T(k+1)  - \tilde{\bar{x}}_{j}^T(k) ] vecs(\bar{P}^i) \\
    =& - \bar{x}_{j}^T(k) \big[\frac{N-1}{N}(\Gamma^TQ\Gamma-Q\Gamma-\Gamma^TQ)  
         + {{} {\Lambda}_{2} }^{T} (R + {\Lambda}_{1} )^{-1} {\Lambda}_{2}  \\
        &- (\bar{K} ^{i})^T (R + {\Lambda}_{1} ) \bar{K} ^{i} \big]\bar{x}_{j}(k)  
         + (\tilde{\bar u}_j(k)+\widetilde{{\bar K} ^{j } {\bar x}_j(k)})  ^T vecs(\bar{\Lambda}_{1}^{i}) \\
        &+ 2 ({\bar x}_{j}^T(k) \otimes {\bar u}^T_{j}(k)) vec(\bar{\Lambda}_{2}^{i}).
\end{aligned}}  \nonumber
\end{equation} 
 
Based on above equations, we can design the model-free algorithm and analyze its convergence (Theorems \ref{thm7} and \ref{thm8}).
For the infinite-horizon model-free algorithm,
the new data matrices need to be defined as follows
\begin{equation} \medmath{
    \left\{ 
        \begin{aligned}
        & \mathcal{N}_{x_{j}} \triangleq \Big[ I^{1},I^{2},\cdots,I^{l} \Big]^{\mathrm{T}},I\triangleq \mathbb{E}( \tilde{x}_{j}(k) - \tilde{x}_{j}(k+1) ),\\
        & \Xi \triangleq \Big[ \xi^{1},\xi^{2},\cdots,\xi^{l} \Big]^{\mathrm{T}},\\
        & \xi \triangleq \mathbb{E} \Big\{ x^T_{j} (k) \Big[ Q +\frac{1}{N}(\Gamma^TQ\Gamma-Q\Gamma-\Gamma^TQ) \\
        & +{(K^{i})}^TRK^{i} \Big] x_{j} (k) \Big\},\\
        &\mathcal{I}_{x_{j}} \triangleq \Big[ I_{x_{j} x_{j}}^{ 1},I_{x_{j} x_{j}}^{ 2},\cdots,I_{x_{j} x_{j}}^{ l} \Big]^{\mathrm{T}},I_{x_{j} x_{j}} \triangleq \mathbb{E}(\tilde{x}_{j}(k)),\\
        &\mathcal{I}_{u_{j}} \triangleq \Big[ I_{u_{j} u_{j}}^{ 1},I_{u_{j} u_{j}}^{ 2},\cdots,I_{u_{j} u_{j}}^{ l} \Big]^{\mathrm{T}},I_{u_{j} u_{j}} \triangleq \mathbb{E}(\tilde{u}_{j}(k)),\\
        &\mathcal{I}_{x_{j} u_{j}K} \triangleq \Big[ I_{x_{j} u_{j}K}^{ 1},I_{x_{j} u_{j}K}^{ 2},\cdots,I_{x_{j} u_{j}K}^{ l} \Big]^{\mathrm{T}},\\
        &I_{x_{j} u_{j}K} \triangleq \mathbb{E} \Big\{x_{j}^T(k) \otimes (u_{j}(k)-K^{i} x_{j}(k))^T \Big\}^T,\\
        &\mathcal{I}_{u_{j}x_{j} } \triangleq \Big[ I_{u_{j}x_{j}  }^{ 1},I_{u_{j} x_{j}}^{ 2},\cdots,I_{u_{j}x_{j} }^{ l} \Big]^{\mathrm{T}},\\
        &I_{u_{j}x_{j} } \triangleq \mathbb{E} \Big\{  \tilde{u}_{j}(k) - \widetilde{K^{i} x_{j}(k)} \Big\},\\
        &  {\mathcal{I}}_{ {x}_{j}  {u}_{j} } \triangleq \Big[  {I}_{ {x}_{j}  {u}_{j} }^{ 1},  {I}_{ {x}_{j}  {u}_{j} }^{ 2},\cdots,  {I}_{ {x}_{j}  {u}_{j} }^{ l} \Big]^{\mathrm{T}}, \\
        &  {I}_{ {x}_{j}  {u}_{j} } \triangleq  \Big\{  {x}^T_j(k) \otimes   {u}^T_j(k)  \Big\}^T,          
    \end{aligned}
    \right.} \nonumber
\end{equation}

and
   \begin{equation}\medmath{
        \left\{ 
            \begin{aligned}
            &\mathcal{N}_{\bar{x}}  \triangleq \Big[ \bar{I}^{1},\bar{I}^{2},\cdots,\bar{I}^{l} \Big]^{\mathrm{T}},
            \bar{I}\triangleq  \tilde{\bar{x}}_{j}(k)  -  \tilde{\bar{x}}_{j}(k+1) ,\\
            &\bar{\Xi}  \triangleq \Big[\bar{\xi}^{1},\bar{\xi}^{2},\cdots,\bar{\xi}^{l} \Big]^{\mathrm{T}},\\
            &\bar{\xi} \triangleq \bar{x}_j^T(k) \Big[ \frac{N-1}{N}(\Gamma^TQ\Gamma-Q\Gamma-\Gamma^TQ) \\
            &+ {{} {\Lambda}_{2} }^{T} (R + {\Lambda}_{1} )^{-1} {\Lambda}_{2}   
             - (\bar{K} ^{i})^T (R + {\Lambda}_{1} ) \bar{K} ^{i}  \Big]\bar{x}_j(k), \\ 
        & \bar{\mathcal{I}}_{\bar{x}} \triangleq \Big[ \bar{I}_{\bar{x}}^{ 1}, \bar{I}_{\bar{x}}^{ 2},\cdots, \bar{I}_{\bar{x}}^{ l} \Big] ^{\mathrm{T}}, 
        \bar{I}_{\bar{x}} \triangleq  \tilde{\bar{x}}_j(k) ,\\        
        & \bar{\mathcal{I}}_{\bar{u}} \triangleq \Big[ \bar{I}_{\bar{u}}^{ 1}, \bar{I}_{\bar{u}}^{ 2},\cdots, \bar{I}_{\bar{u}}^{ l} \Big] ^{\mathrm{T}}, 
        \bar{I}_{\bar{u}} \triangleq  \tilde{\bar{u}}_j(k) ,\\ 
        & \bar{\mathcal{I}}_{ \bar{K} \bar{x} \bar{u}} \triangleq \Big[ \bar{I}_{ \bar{K} \bar{x}\bar{u}}^{ 1}, \bar{I}_{ \bar{K} \bar{x}\bar{u}}^{ 2},\cdots, \bar{I}_{ \bar{K} \bar{x}\bar{u}}^{ l} \Big]^{\mathrm{T}}, \\
        & I_{\bar{K} \bar{x}\bar{u}} \triangleq  \widetilde{   \bar{K}_{k}^{i} \bar{x}_j (k)} + \tilde{\bar{u}}_j(k)   ,\\        
        & \bar{\mathcal{I}}_{\bar{x} \bar{u}} \triangleq \Big[ \bar{I}_{\bar{x} \bar{u}}^{ 1}, \bar{I}_{\bar{x} \bar{u}}^{ 2},\cdots, \bar{I}_{\bar{x} \bar{u}}^{ l} \Big]^{\mathrm{T}}, \\
        & \bar{I}_{\bar{x} \bar{u}} \triangleq  \Big\{ \bar{x}^T_j(k) \otimes  \bar{u}^T_j(k)  \Big\}^T.               
            \end{aligned}\nonumber
            \right.}  
    \end{equation} 

\begin{assum} \label{ass1_infinite}
    There exists an $l_{1} >0$, such that, for all $l \geq l_{1}$, 
     \begin{equation}
        \text{rank}\left(\begin{bmatrix}
        \mathcal{I}_{x_{j}}, \,\, \mathcal{I}_{x_{j} u_{j} } , \,\, \mathcal{I}_{u_{j}}     
        \end{bmatrix}\right)
        =\frac{n}{2}(n+1)+mn+\frac{m}{2}(m+1). \nonumber
        \end{equation} 
\end{assum}

\begin{thm}  \label{thm7}
    Suppose Assumption \ref{ass1_infinite} holds.  $P$ is the positive semi-definite solution to the ARE \eqref{section4_1}.
    The sequence $\{ P^i, \Lambda_1^{i}, \Lambda_2^{i}  \}_{i=1}^{\infty}$  
    generated by recursively solving the following equation
    \begin{equation}
    \begin{bmatrix} vecs(P^i)   \\ vec(\Lambda_2^{i})  \\ vecs(\Lambda_1^{i}) \end{bmatrix}
    = (\Psi^T \Psi)^{-1} \Psi^T \Xi, \label{infinite_1_model_free_theorem}
\end{equation}
satisfies $\lim\limits_{ i \to \infty } P^i = P $,  
where 
$\Psi  \triangleq [\mathcal{N}_{x_{j}}, 2\mathcal{I}_{x_{j} u_{j}K}, \mathcal{I}_{u_{j}x_{j} }]$. 
All the control gains $K^{i+1} = -(R+\Lambda_{1}^{i})^{-1}\Lambda_{2}^{i} $ are  stabilizers.
\end{thm}

\emph{Proof.} 
The iteration is initialized with a stabilizer $K^0$. 
Assuming  $K^i$ is a stabilizer, we prove that $K^{i+1}$ is also a stabilizer.
Define a nonzero vector $H=[X_v^T,Y_v^T,Z_v^T]^T$, 
where 
$X_v= vecs(X) \in \mathbb{R}^ {\frac{n}{2}(n+1)}, $ 
$Y_v=vec(Y) \in \mathbb{R}^{mn},  $
$Z_v= vecs(Z) \in \mathbb{R}^{\frac{m}{2}(m+1)} $.
From \eqref{e1_1}, we obtain
\begin{equation}
    \begin{aligned}
 \mathbb{E}&[\tilde{x}^T_{j}(k+1) X_v] = \mathbb{E} \big[ \tilde{x}^T_{j}(k) vecs(A^T X A + \sigma^2 C^T X C) \\
 &+ 2 (x_j^T(k) \otimes u_j^T(k)) vec(B^T X A + \sigma^2 D^T X C) \\
 &+ \tilde{u}^T_{j}(k) vecs(B^T X B + \sigma^2 D^T X D) \big].
    \end{aligned} \nonumber
\end{equation}
Considering the equation $\Psi H - \Xi = 0$, one has
\begin{equation}
    \begin{aligned}
 \Big[\mathcal{I}_{x_{j}},\, \mathcal{I}_{x_{j} u_{j}},\, \mathcal{I}_{u_{j}} \Big] \begin{bmatrix} vecs(M_1)   \\ vec(M_2)  \\ vecs(M_3) \end{bmatrix}
    = 0,
    \end{aligned} \label{infinite_1_model_free}
\end{equation} 
where 
\begin{equation}\medmath{
    \left\{
    \begin{aligned}
        M_1 &= X - {(K^{i})}^TY-Y^T{K^{i}} -{(K^{i})}^TZ{K^{i}} \\
        &-A^T X A - \sigma^2 C^T X C - \Big[ Q +\frac{1}{N}(\Gamma^TQ\Gamma-Q\Gamma \\
        & -\Gamma^TQ)+{(K^{i})}^TRK^{i} \Big], \\
        M_2 &= 2(Y - B^T X A - \sigma^2 D^T X C),  \\
        M_3 &= Z - B^T X B - \sigma^2 D^T X D.    
    \end{aligned}    
    \right.} \nonumber
\end{equation}
Since Assumption \ref{ass1_infinite} holds,
the equation \eqref{infinite_1_model_free} has a unique solution. Therefore, we have $ {M}_1 = 0,  {M}_2 = 0,  {M}_3 = 0$.
Substituting $Y = B^T X A + \sigma^2 D^T X C$, $Z = B^T X B + \sigma^2 D^T X D$ into $M_1=0$, one has
  \begin{equation}\medmath{
    \begin{aligned}
      X&=  (A + B K^{i})^T X (A + B K^{i})  \\
      &+ \sigma^2 (C + D K^{i})^T X (C + D K^{i}) \\
      &+ Q +\frac{1}{N}(\Gamma^TQ\Gamma-Q\Gamma  -\Gamma^TQ)+{(K^{i})}^TRK^{i} . \nonumber
    \end{aligned}}
\end{equation}
Let $X = P^{i}$, $Y = \Lambda_1^{i}$ and $Z = \Lambda_2^{i}$. Then the equation $\Psi H - \Xi = 0$ or the above equation becomes \eqref{infinite_1_model_free_theorem},
which is the same as \eqref{th4_1}. 
The control gain is updated as $ K^{i+1} = -(R+\Lambda_{1}^{i})^{-1}\Lambda_{2}^{i} $, 
which coincides with \eqref{th4_2}.  
Using the Lyapunov-type operator and \eqref{section4_1}, we have  
$\mathcal{L}_{i} (P^{i} - P ) - (P^{i} - P ) = - (K^{i}-K )^T (R+B^T P B +\sigma^2 D^T P D)(K^{i}-K )$,
which implies that $P^{i} - P  \geq 0$.
From \eqref{section4_7}, one has
\begin{equation}
    \begin{aligned} 
        \mathcal{L}_{i+1} &(P^{i} - P^{i+1}) - (P^{i} - P^{i+1}) = - (K^{i}-K^{i+1})^T \\
        &\times (R+B^TP^{i}B +\sigma^2 D^TP^{i}D)(K^{i}-K^{i+1}). 
    \end{aligned} \nonumber
\end{equation}
Assume there exists $\lambda \in \sigma(\mathcal{L}_{i + 1}(\mathcal{X}))$ such that $\lambda \geq 1$.
Then, we have
\begin{equation}
    \begin{aligned} 
        & (\lambda - 1 ) Tr [(P^{i} - P^{i+1})\mathcal{X}] = - Tr  [ \mathcal{X}^{1/2}{(K^{i}-K^{i+1})}^T \\
        & \times (R  + B^TP^{i}B + \sigma^2 D^TP^{i}D) {(K^{i}-K^{i+1})} \mathcal{X}^{1/2} ].
    \end{aligned} \nonumber
\end{equation}
The above equation implies that $(K^{i}-K^{i+1}) \mathcal{X}=0$. 
Consequently, this indicates that $\lambda \geq 1$ is also in the spectrum of the previous operator $ \sigma(\mathcal{L}_{i}(\mathcal{X}))$.
It contradicts with the induction assumption.
Therefore, the gain $K^{i+1}$ updated by \eqref{infinite_1_model_free_theorem} is a stabilizer for all $i \in \mathbb{N}$.
By Theorem \ref{thm5}, the convergence is proved.
{{\hfill{$\Box$}}}

\begin{assum} \label{ass2_infinite}
    There exists an $l_2 > 0$, such that, for all $l \geq l_2$, 
    \begin{equation}
        \text{rank}\left(\begin{bmatrix}
            \bar{\mathcal{I}}_{\bar{x}}, \,\, \bar{\mathcal{I}}_{\bar{x} \bar{u}}, \,\, \bar{\mathcal{I}}_{\bar{u}}   
        \end{bmatrix}\right)
    =\frac{n}{2}(n+1)+ mn+\frac{m}{2}(m+1). \nonumber
    \end{equation} 
    \end{assum}

    \begin{thm}  \label{thm8}
        Suppose Assumption \ref{ass2_infinite} holds. $\bar{P}$ is the positive semi-definite solution to the ARE \eqref{section4_2}. 
        The sequence $\{ S^i, \bar{\Lambda}_1^{i}, \bar{\Lambda}_2^{i}  \}_{i=1}^{\infty}$  generated by recursively solving the following equation
    \begin{equation}
        \begin{bmatrix} vecs(\bar{P}^i)   \\ vec( \bar{\Lambda}_2^{i})  \\ vecs( \bar{\Lambda}_1^{i}) \end{bmatrix}
        =(\bar{\Psi}^T \bar{\Psi})^{-1} \bar{\Psi}^T \bar{\Xi}, \label{infinite_2_model_free_theorem}
    \end{equation} 
    \end{thm}
    satisfies $\lim\limits_{ i \to \infty } \bar{P}^i = \bar{P} $,  
    where 
    $\bar{\Psi}  \triangleq [\mathcal{N}_{\bar{x}}, 2 \bar{\mathcal{I}}_{ \bar{x} \bar{u} },  \bar{\mathcal{I}}_{\bar{K} \bar{x} \bar{u}}]$.
    All the control gains $\bar{K}^{i+1} = -(R + {\Lambda}_{1}  + \bar{\Lambda}_{1}^i )^{-1} ( {\Lambda}_{2}  + \bar{\Lambda}_{2}^i )$ are stabilizers.

    \emph{Proof.} 
Define a nonzero vector $\bar{H} = [\bar{H}_{v,1}^T,\bar{H}_{v,2}^T,\bar{H}_{v,3}^T]^T$, 
where 
$\bar{H}_{v,1}= vecs(\bar{H}_1) \in \mathbb{R}^ {\frac{n}{2}(n+1)}, $ 
$\bar{H}_{v,2}=vec(\bar{H}_2) \in \mathbb{R}^{\frac{m}{2}(m+1)},  $
$\bar{H}_{v,3}= vecs(\bar{H}_3) \in \mathbb{R}^{mn} $.
From \eqref{e2_4}, we obtain
\begin{equation}
    \begin{aligned}
  & \tilde{\bar{x}}^T_{j}(k+1) \bar{H}_{v,1} =  \tilde{x}^T_{j}(k) vecs(A^T \bar{H}_1 A  ) \\
 &+ 2 (x_j^T(k) \otimes u_j^T(k)) vec(B^T \bar{H}_1 A ) \\
 &+ \tilde{u}^T_{j}(k) vecs(B^T \bar{H}_1 B ) .
    \end{aligned} \nonumber
\end{equation}
Consider the equation $\bar{\Psi}  \bar{H} - \bar{\Xi} = 0$, one has
\begin{equation}
    \begin{aligned}
 \Big[\bar{\mathcal{I}}_{\bar{x}}, \, \bar{\mathcal{I}}_{\bar{x} \bar{u}}, \, \bar{\mathcal{I}}_{\bar{u}}  \Big] 
 \begin{bmatrix} vecs(\bar{M}_{ 1})   \\ vec(\bar{M}_{ 2})  \\ vecs(\bar{M}_{ 3}) \end{bmatrix}
    = 0,
    \end{aligned} \label{infinite_2_model_free}
\end{equation} 
where 
\begin{equation}\medmath{
    \left\{
    \begin{aligned}
        \bar{M}_{ 1} &= \bar{H}_{ 1} -A^T \bar{H}_{ 1} A +{(\bar{K}^{i})}^T \bar{H}_{ 3} {\bar{K}^{i}} \\
        &  - \Big[ \frac{N-1}{N}(\Gamma^TQ\Gamma-Q\Gamma-\Gamma^TQ) \\
        & + { {\Lambda}_{2} }^{T} (R + {\Lambda}_{1} )^{-1} {\Lambda}_{2}- {(\bar{K}^{i})}^T (R + {\Lambda}_{1} ) \bar{K}^{i} \Big], \\
        \bar{M}_{ 2} &= 2(\bar{H}_{ 2} - B^T \bar{H}_{ 1} A  ),  \\
        \bar{M}_{ 3} &= \bar{H}_{  3} - B^T \bar{H}_{ 1} B  . 
    \end{aligned}    
    \right.} \nonumber
\end{equation}
Since Assumption \ref{ass2_infinite} holds,
the equation \eqref{infinite_2_model_free} has a unique solution. Therefore, we have $ \bar{M}_1 = 0,  \bar{M}_2 = 0,  \bar{M}_3 = 0$.
Substituting $ \bar{H}_{  3} = B^T \bar{H}_{ 1} B$ into $\bar{M}_1=0$, one has
  \begin{equation}\medmath{
    \begin{aligned}
      \bar{H}_{ 1} &=   \Big[ \frac{N-1}{N}(\Gamma^TQ\Gamma-Q\Gamma-\Gamma^TQ)  + { {\Lambda}_{2} }^{T} (R + {\Lambda}_{1} )^{-1} {\Lambda}_{2} \Big] \\
        &+ A ^T \bar{H}_{ 1}  A  - {(\bar{K}^{i})}^T (R + {\Lambda}_{1} + B^T \bar{H}_{ 1}  B ) \bar{K}^{i}   , \label{infinite_2_model_free_pf2}
    \end{aligned}}
\end{equation} 
Let $\bar{H}_{ 1}  = \bar{P}^{i}$, the above equation becomes \eqref{infinite_2_model_free_theorem}
and can also be obtained by subtracting \eqref{section4_1} from \eqref{section4_8}.
Similar to the proof of Theorem \ref{thm7},
the gain $\bar{K}^{i+1}$ updated by \eqref{infinite_2_model_free_theorem} is a stabilizer for all $i \in \mathbb{N}$.
By Theorem \ref{thm6}, the convergence is proved and the sequence in \eqref{infinite_2_model_free_theorem} or \eqref{infinite_2_model_free_pf2} is asymptotically equivalent to $\{\bar{P}^{i}\}_{i=0}^\infty$.
{{\hfill{$\Box$}}}

    \begin{thm}  \label{thm9}
        Assume the initial control gains are set as 
        $K_{k}^{ 0 } = K^{ 0 }$ and $\bar{K}_{k}^{ 0 } = \bar{K}^{ 0 }$ 
        for all $k \in [k_0, T)$ in Theorems~\ref{thm1} and~\ref{thm2}, 
        where $K^{ 0 }$ and $\bar{K}^{ 0 }$ are identical to those in 
        Theorems~\ref{thm5} and~\ref{thm6}. 
        Then, as $T \to \infty$, it follows that 
        $\lim_{k \to -\infty} P_{k}^{ i } = P^{i}$, 
        $\lim_{k \to -\infty} K_k^{ i } = K^{ i }$, 
        $\lim_{k \to -\infty} S_{k}^{ i } = S^{i}$, 
        and $\lim_{k \to -\infty} \bar{K}_{k}^{ i } = \bar{K}^{ i }$. 
        Thus, for $T \to \infty$, the finite-horizon policy iteration algorithms in 
        Theorems~\ref{thm1} and~\ref{thm2} become asymptotically equivalent to the infinite-horizon 
        algorithms in Theorems~\ref{thm5} and~\ref{thm6}.
    \end{thm}

\emph{Proof.} 
For the $i$-th iteration of policy iteration and $(\tilde{P}^i, \tilde{S}^i) \in \bar{\mathcal{M}}$, we consider the auxiliary cost function
\begin{equation}
\begin{aligned}
\mathcal{J}^i(l, \zeta) &= \sum_{k=l}^T \mathbb{E} \Big\{  \begin{bmatrix}
\bar{y}_k \\
\bar{v}_k
\end{bmatrix}^T
\overline{\mathcal{H}} (\tilde{P}^i, \tilde{S}^i)
\begin{bmatrix}
\bar{y}_k \\
\bar{v}_k
\end{bmatrix} \\
&+
\begin{bmatrix}
y_k - \bar{y}_k \\
v_k - \bar{v}_k
\end{bmatrix}^T
\mathcal{H} (\tilde{P}^i)
\begin{bmatrix}
y_k - \bar{y}_k \\
v_k - \bar{v}_k
\end{bmatrix}
\Big\},
\end{aligned} \label{payoff_MF}
\end{equation}
with initial state $\zeta$. 
Furthermore, the corresponding value function is
\begin{equation}
\begin{aligned}
\mathcal{V} ^i(l) &=
\inf_{u \in\mathcal{U}_{ad} } \mathcal{J}^i(l, \zeta) \\
&=
\mathbb{E}\left[(y_l - \bar y_l)^T U_l^i (y_l - \bar y_l)\right] 
+ (\bar y_l)^T V_l^i \bar y_l. 
\end{aligned} \label{Value_MF}
\end{equation}
We first prove the monotonicity of the sequences $\{U_l^i\}_{l=-\infty}^T$ and $\{V_l^i\}_{l=-\infty}^T$ with respect to time horizon  $l$ .
Since $(\tilde{P}^i, \tilde{S}^i) \in \bar{\mathcal{M}}$, we have
$\overline{\mathcal{H}} (\tilde{P}^i, \tilde{S}^i) \geq 0$ and
$ \mathcal{H} (\tilde{P}^i ) \geq 0$.
Consequently, every term in \eqref{payoff_MF} is nonnegative.
Extending the time horizon backwards by one step (from $[l,T]$ to $[l-1,T]$) cannot decrease the cost function for the same initial state $\zeta$.
Therefore, $\mathcal{J}^i(l-1, \zeta) \geq \mathcal{J}^i(l, \zeta)$ and $\mathcal{V} ^i(l-1) \geq \mathcal{V} ^i(l)$.
To prove the monotonicity of the sequences $\{U_l^i\}_{l=-\infty}^T$ and $\{V_l^i\}_{l=-\infty}^T$ separately,
as in \cite{ni2015indefinite},
we construct  the following two specific choices of the initial state $\zeta$:

(i) Let $\zeta = \mu \in \mathbb{R}^n$, 
from \eqref{Value_MF}, we have $\mu^T V_{l-1}^{i} \mu \geq  \mu^T V_l^{i} \mu$, which implies $ V_{l-1}^{i} \geq  V_l^{i}$.

(ii) Let $\zeta = \xi \epsilon,$ $ \xi \in \mathbb{R}^n$, $\epsilon$ is a random variable with $\mathbb{P}(\epsilon = 1) = \mathbb{P}(\epsilon = -1) = \frac{1}{2}$.  
From \eqref{Value_MF}, we have $\xi^T U_{l-1}^{i} \xi \geq  \xi^T U_l^{i} \xi$, which implies $ U_{l-1}^{i} \geq  U_l^{i}$.
 
Next, we show that the statement of the theorem holds for all $i \in \mathbb{N}$ by mathematical induction.
If $i = 0$, by Assumption \ref{assum_stab}, the system \eqref{e1_1} is $L^2-$stabilizable 
and for any stabilizing control $v_k = K (y_k - \bar{y}_k) + \bar{K} \bar{y}_k, $
the state $y_k$  is mean-square bounded.
Since 
$(\mathbb{E}  (y))^2  \leq \mathbb{E} (y^2) $ and 
$\mathbb{E}  [y - \mathbb{E}(y)]^2 = \mathbb{E} (y^2) - (\mathbb{E}(y))^2 \leq \mathbb{E} (y^2) $,
\eqref{Value_MF} becomes
\begin{equation}
  \begin{aligned}
  &  \mathbb{E}\left[(y_l - \bar y_l)^T U_l^0 (y_l - \bar y_l)\right] + (\bar y_l)^T V_l^0 \bar y_l \\
& \leq 
    \sum_{k=l}^T \mathbb{E} \Big\{ \bar{y}_k^T \begin{bmatrix}
I_n \\
\bar{K}
\end{bmatrix}^T
\overline{\mathcal{H}} (\tilde{P}^0, \tilde{S}^0)
\begin{bmatrix}
I_n \\
\bar{K}
\end{bmatrix} \bar{y}_k \\
&+ (y_k - \bar{y}_k)^T
\begin{bmatrix}
I_n \\
{K}
\end{bmatrix}^T
\mathcal{H} (\tilde{P}^0)
\begin{bmatrix}
I_n \\
{K}
\end{bmatrix} (y_k - \bar{y}_k)
\Big\}\\
& \leq
 c \sum_{k=l}^{\infty} \mathbb{E}|y_k|^2 < \infty,
  \end{aligned}  
\end{equation}
where $c > 0$ is a constant.
Next, we need to discuss the boundedness of the sequences $\{U_l^0\}_{l=-\infty}^T$ and $\{V_l^0\}_{l=-\infty}^T$ separately.
 For deterministic $y_l = \mu$, we have $\mu^T V_l^{0} \mu < \infty$ from the above equation.
 For random $y_l = \xi \epsilon,$ we have $\xi^T U_l^{0} \xi < \infty$ from the above equation.

Thus as $l \to -\infty$, 
$\{U_l^0,U_{l-1}^0,U_{l-2}^0, \cdots, U_{-\infty}^0 \}$ 
and  $\{V_l^0,V_{l-1}^0,V_{l-2}^0, \cdots, V_{-\infty}^0 \}$ are bounded nondecreasing sequences.
Let  $P_k^i = U_k^i + \tilde P^i$ and $S_k^i = V_k^i + \tilde S^i$.
$\lim_{k \to -\infty} P_{k}^{ 0 } $ and $\lim_{k \to -\infty} S_{k}^{ 0 } $ exist.
Taking the limit $k \to -\infty$ on the both sides of \eqref{th1_1} and \eqref{th2_1},
$(P_{k}^{0}, K_{k}^{0}, S_{k}^{0}, \bar{K}_{k}^{0})$ converges to $(P^{0}, K^{0}, S^{0}, \bar{K}^{0})$, where $(P^{0} , S^{0})$ is the solution given in \eqref{th4_1} and \eqref{section4_8} when $i=0$.

Suppose the result holds for $i > 0$, that is, $(P_{k}^{i}, K_{k}^{i}, S_{k}^{i}, \bar{K}_{k}^{i})$ converges to $(P^{i}, K^{i}, S^{i}, \bar{K}^{i})$ as $k \to -\infty$.
We prove that the result also holds for $i+1$.
By the induction hypothesis, 
$(K_{k}^{i+1}, \bar{K}_{k}^{i+1})$ converges to $(K^{i+1}, \bar{K}^{i+1})$ as $k \to -\infty$, 
where $K_{k}^{i+1}$ and $\bar{K}_{k}^{i+1}$ are given in \eqref{th1_2} and \eqref{th2_2}, respectively,
and $K^{i+1}$, $\bar{K}^{i+1}$  are given in  \eqref{th4_2} and \eqref{section4_9}, respectively.
Based on Theorems \ref{thm1} and \ref{thm2}, $P_{k}^{i+1} \leq P_{k}^i$ and $S_{k}^{i+1} \leq S_{k}^i$. 
By the induction hypothesis, that is, $P_{k}^i < \infty$ and $S_{k}^i < \infty$,  
as $k \to -\infty$, 
$\{P_{k}^{i+1},P_{k-1}^{i+1},P_{k-2}^{i+1}, \cdots, P_{-\infty}^{i+1} \}$ 
and  $\{S_{k}^{i+1},S_{k-1}^{i+1},S_{k-2}^{i+1}, \cdots, S_{-\infty}^{i+1} \}$ are bounded nondecreasing sequences. 
The convergence of $\{P_{k}^{i+1}\}_{k=T}^{-\infty}$ and $\{S_{k}^{i+1}\}_{k=T}^{-\infty}$ are proved.  
Furthermore,  $P^{i+1}$ and $S^{i+1}$ satisfy  \eqref{th1_1} and \eqref{th2_1} with $K_{k}^{i+1} =K^{i+1}$ and $\bar{K}_{k}^{i+1} = \bar{K}^{i+1}$, respectively.
Thereby, $\lim_{k \to -\infty} P_{k}^{i+1} = P^{i+1} $ and $\lim_{k \to -\infty} S_{k}^{i+1} = S^{i+1} $. 
Hence, the proof is complete.
{{\hfill{$\Box$}}}

Next, we compare the finite-horizon and infinite-horizon cases.
The main differences are listed below.

    \begin{rem}
        The main distinctions between infinite-horizon and finite-horizon cases are as follows: \\
        (1) Stability requirement. In the infinite-horizon case, an initial stabilizing gain and system stability are required. 
        In contrast, the finite-horizon case does not require stability.   \\
        (2) Time-variety of control policies. In the infinite-horizon case, the control strategies are time-invariant, 
        whereas in the finite-horizon case, 
        it is time-varying since the solution is computed backward from the terminal time $T+1$ to the initial time. \\
        (3) Data utilization efficiency. In the infinite-horizon case, 
        data is efficiently utilized to learn a single fixed policy, 
        whereas the finite-horizon case requires collecting time-specific data for each step 
        due to the policy's time-dependency, increasing complexity.
        In finite-horizon case,
        optimizing $T+1$ strategies (one per time step) leads to a significant increase in parameter space and complexity.
    \end{rem}

To facilitate the robustness analysis, we first introduce the following notation:
$ \beta_1(i) \triangleq (1 - \frac{1}{a_1})^{i-1} \sqrt{n}$,  
$ \beta_2 \triangleq a_1    a_0(\varepsilon) \cdot a_2(\varepsilon) $,
$a_0(\varepsilon) \triangleq \sup_{\hat{K}_i \in \mathcal{G}} \lVert \sum_{t=0}^{\infty} (\hat{\mathcal{L}}_i^*)^t(I) \rVert_2  $,
$a_1 \triangleq \lVert \sum_{t=0}^{\infty}  {({\mathcal{L}}^*)}^t (I) \rVert_2$,
$a_2(\varepsilon) \triangleq \sup_{\hat{K}_i \in \mathcal{G}} \lVert (R + B^T \hat{P}^{i-1} B + \sigma^2 D^T \hat{P}^{i-1} D) \rVert_2$,
$ \bar \beta_1(i) \triangleq (1 - \frac{1}{\bar a_1})^{i-1} \sqrt{n}$,  
$ \bar \beta_2 \triangleq \bar a_1    \bar a_0(\varepsilon) \cdot \bar a_2(\varepsilon) $,
$ \bar a_0(\varepsilon) \triangleq \sup_{\hat{\bar K}_i \in \bar {\mathcal{G}}} $ $\lVert \sum_{t=0}^{\infty} (\hat{\bar {\mathcal{L}}}_i^*)^t(I) \rVert_2  $,
$ \bar a_1 \triangleq \lVert \sum_{t=0}^{\infty}  {(\bar {\mathcal{L}}^*)}^t (I) \rVert_2$,
and
$ \bar a_2(\varepsilon) \triangleq \sup_{\hat{\bar K}_i \in \bar {\mathcal{G}}} \lVert (R + B^T \hat{S}^{i-1} B + \sigma^2 D^T P D) \rVert_2$.
The $l^\infty$-norm is defined as
$\lVert \Delta K \rVert_{\infty} \triangleq \sup_{i \in \mathbb{Z}_+} \lVert \Delta K_i \rVert_F$.
The operator $ {\mathcal{L}}^*$ is the adjoint operator of $ {\mathcal{L}} $.

    \begin{thm}  \label{thm10}     
Given $\varepsilon > 0$ and $\hat{K}^1 \in \mathcal{G}$,
where $\mathcal{G} \triangleq \{ \hat{K}^i |\hat{K}^i \, \text{is a stabilizer} , Tr(\hat{P}^i) - Tr(P) \leq \varepsilon \}$,  
there exists $\delta >0$ such that if $\lVert \Delta K \rVert^2_{\infty} < \delta$,
then we have  
\begin{enumerate} 
    \item[(i)]  the following small-disturbance input-to-state stability (ISS) holds:
\begin{equation}
    \begin{aligned} 
        \lVert \hat{P}^{i} - {P}  \rVert_F   
        &\leq  \beta_1(i)  \lVert \hat{P}^{1} - {P}  \rVert_F  
        + \beta_2  \lVert {\Delta K}  \rVert_{\infty}^2. \\
    \end{aligned} \label{ISS_1}
\end{equation}
    \item[(ii)] $\lim_{i \to \infty} \|\Delta  K_i \|_F = 0 \text{ implies } \lim_{i \to \infty} \|\hat{P}_i - P \|_F = 0.$
\end{enumerate}     
\end{thm}

\emph{Proof.} 
We show that if the current control gain is a stabilizer, 
the next gain remains a stabilizer even under suitable perturbations. 
Given that $\hat{K}^{i-1}$ is a stabilizer, 
there exists a threshold $\delta(r) > 0$ such that 
for any $\hat{P}^{i-1}$ within the closed ball $\bar{\mathcal{B}}(P, r)$, 
the condition $\|\Delta K_{i}\|^2_F < \delta(r)$ ensures that 
the updated controller $\hat{K}^{i}$ remains a stabilizer  
and the matrix $(R+B^T \hat{P}^{i-1}B+\sigma^2 D^T \hat{P}^{i-1}D)$ is invertible \cite[Lemma 4]{pang2021robust}. 
Substituting the perturbed controller $\hat{K}^i = {K}^i + \Delta {K}_i$ into the analysis, 
the original equation \eqref{th4_1} is rewritten as
\begin{equation}\medmath{
    \begin{aligned} 
        \hat{P}^i= & Q+\frac{1}{N}(\Gamma^TQ\Gamma-Q\Gamma-\Gamma^TQ) \\
        & +{(\hat{A}^i)}^T \hat{P}^i \hat{A}^i  
          +\sigma^2{(\hat{C}^i)}^T \hat{P}^i \hat{C}^i+{(\hat{K}^i)}^T R \hat{K}^i. 
    \end{aligned}}  \label{robust1}
\end{equation}
The corresponding Lyapunov-type operator is $\hat{\mathcal{L}}_i(X)$. 
Subtracting the above equation from ${(\hat{A}^i)}^T \hat{P}^{i-1} \hat{A}^i +\sigma^2{(\hat{C}^i)}^T \hat{P}^{i-1} \hat{C}^i$ 
and substituting $\hat{K}^i = {K}^i + \Delta {K}_i$ yields
\begin{equation} \medmath{
    \begin{aligned} 
        &{(\hat{A}^i)}^T ( \hat{P}^{i-1} - \hat{P}^{i}) \hat{A}^i  +\sigma^2{(\hat{C}^i)}^T ( \hat{P}^{i-1} - \hat{P}^{i}) \hat{C}^i - ( \hat{P}^{i-1} - \hat{P}^{i}) \\
        & = -{( {K}^i - \hat{K}^{i-1})}^T \hat{\Upsilon}^{i-1} ( {K}^i - \hat{K}^{i-1})  
        + {\Delta K}_i^T \hat{\Upsilon}^{i-1} {\Delta K}_i,
    \end{aligned}} \label{robust2}
\end{equation}
where $\hat{\Upsilon}^{i-1} =  {\Upsilon}(\hat{P}^{i-1})  \triangleq  R + B^T \hat{P}^{i-1} B + \sigma^2 D^T \hat{P}^{i-1} D $.
Since $\hat{K}^i$ is stabilizer, 
the difference equation above admits a unique series solution
\begin{equation}\medmath{
    \begin{aligned} 
        \hat{P}^{i-1} - \hat{P}^{i}  
        & = \sum_{t=0}^{\infty} \hat{\mathcal{L}}_i^t [ {( {K}^i - \hat{K}^{i-1})}^T \hat{\Upsilon}^{i-1}( {K}^i - \hat{K}^{i-1}) 
        \\& - {\Delta K}_i^T \hat{\Upsilon}^{i-1} {\Delta K}_i].
    \end{aligned}} \label{robust3}
\end{equation}

Take the trace of both sides of \eqref{robust3}.
Regarding the first term, the matrix ${\Upsilon}( {P}^{i-1})$ is positive definite. 
By the sign-preserving property of continuous functions,
the perturbed matrix $\hat{\Upsilon}^{i-1}$ is positive semi-definite for sufficiently small perturbations. 
We derive the following inequality
\begin{equation}
    \begin{aligned} 
        Tr (\hat{P}^{i-1} - \hat{P}^{i} )  
        & \geq Tr \Big\{ 
           {( {K}^i - \hat{K}^{i-1})}^T \hat{\Upsilon}^{i-1} ( {K}^i - \hat{K}^{i-1})    
        \\& - a_0(\varepsilon) \cdot Tr[{\Delta K}_i^T \hat{\Upsilon}^{i-1} {\Delta K}_i]  
        \Big\},
    \end{aligned} \label{robust4}
\end{equation}
where $\hat{\mathcal{L}}_i^*$ is the adjoint operator of $\hat{\mathcal{L}}_i$, and
$a_0(\varepsilon) \triangleq \sup_{\hat{K}_i \in \mathcal{G}} \lVert \sum_{t=0}^{\infty} (\hat{\mathcal{L}}_i^*)^t(I) \rVert_2  $.
Subtracting \eqref{section4_1} from \eqref{robust1} at the $(i-1)$-th iteration yields  
\begin{equation}\medmath{
    \begin{aligned} 
        \hat{P}^{i-1} - {P}  
         =& \sum_{t=0}^{\infty}  {\mathcal{L}}^t [ {( {K}^{i} - \hat{K}^{i-1} )}^T \hat{\Upsilon}^{i-1} {( {K}^{i} - \hat{K}^{i-1} )}
        \\& - {( {K}^{i} -  {K} )}^T \hat{\Upsilon}^{i-1}( {K}^{i} - {K} )].
    \end{aligned}} \label{robust6}
\end{equation}
Taking the trace of \eqref{robust6} 
and applying the trace inequality 
$\text{Tr}(XY) \leq \text{Tr}(X) \|Y\|_2$, 
we obtain
\begin{equation}\medmath{
    \begin{aligned} 
        Tr(\hat{P}^{i-1} - {P})  
        & \leq a_1 \cdot Tr[{( {K}^{i} - \hat{K}^{i-1} )}^T \hat{\Upsilon}^{i-1} {( {K}^{i} - \hat{K}^{i-1} )} ] ,
    \end{aligned}} \label{robust7}
\end{equation}
where
$a_1 \triangleq \lVert \sum_{t=0}^{\infty}  {({\mathcal{L}}^*)}^t (I) \rVert_2$.
 Since the solution of ARE \eqref{section4_1} is stabilizing, 
the spectrum of ${\mathcal{L}}^*(\cdot)$ lies strictly inside the open unit disc,
yielding $a_1 < 1$. 
Substituting the above inequality into \eqref{robust4} yields 
\begin{equation} \medmath{
    \begin{aligned} 
        - Tr (\hat{P}^{i-1} - \hat{P}^{i} )  
        & \leq - \frac{1}{a_1} Tr(\hat{P}^{i-1} - {P})  + a_0(\varepsilon) \cdot a_2(\varepsilon) \cdot \lVert {\Delta K}_i  \rVert_F^2 , 
    \end{aligned}} \label{robust9}
\end{equation}
where $\medmath{a_2(\varepsilon) \triangleq \sup_{\hat{K}_i \in \mathcal{G}} \lVert (R + B^T \hat{P}^{i-1} B + \sigma^2 D^T \hat{P}^{i-1} D) \rVert_2}$. 
Consequently, we obtain the following inequality
\begin{equation}\medmath{
    \begin{aligned} 
        Tr (\hat{P}^{i} - {P}  )  
        & \leq (1 - \frac{1}{a_1}) Tr(\hat{P}^{i-1} - {P})  + a_0(\varepsilon) \cdot a_2(\varepsilon) \cdot \lVert {\Delta K}_i  \rVert_F^2.  
    \end{aligned}} \label{robust10}
\end{equation}
Let $\delta = \frac{a_1}{a_0(\varepsilon) \cdot a_2(\varepsilon)} \cdot \varepsilon$.
When $\lVert {\Delta K}_i  \rVert_F^2 < \delta$, 
the inequality $Tr (\hat{P}^{i} - {P}  ) < \varepsilon$ holds. 
This implies that the updated perturbed controller $\hat{K}^i \in \mathcal{G}$. 
Letting
$\lVert {\Delta K}  \rVert_{F}^2 = max_{i \in \mathbb{N}} \{ \lVert {\Delta K}_i  \rVert_F^2 \}$,
we have $\lVert {\Delta K}  \rVert_{F}^2 < \delta$.
Repeatedly applying the inequality \eqref{robust10},
one has
\begin{equation}\medmath{
    \begin{aligned} 
        Tr &(\hat{P}^{i} - {P}  )   
        \leq (1 - \frac{1}{a_1})^{i-1} Tr(\hat{P}^1 - {P})  \\
        & + \sum_{j=0}^{i-2} (1 - \frac{1}{a_1})^{j} a_0(\varepsilon) \cdot a_2(\varepsilon) \cdot \lVert {\Delta K}_{i-j}  \rVert_F^2. \\
    \end{aligned}} \label{robust11}
\end{equation}
Since 
$\lVert P  \rVert_F \leq  Tr(P) \leq \sqrt{n} \lVert P  \rVert_F$,
one has
\begin{equation}\medmath{
    \begin{aligned} 
        \lVert \hat{P}^{i} - {P}  \rVert_F  
        & \leq (1 - \frac{1}{a_1})^{i-1} \sqrt{n} \lVert \hat{P}^{1} - {P}  \rVert_F  \\
        &+  a_1 \cdot a_0(\varepsilon) \cdot a_2(\varepsilon) \cdot \lVert {\Delta K}  \rVert_{\infty}^2. \\
    \end{aligned}} \label{robust12}
\end{equation} 
This completes the proof of part (i) of Theorem \ref{thm10}. 

As for (ii) in Theorem \ref{thm10}, 
since $\lim_{i \to \infty} \|\Delta K_i\|_F = 0$, 
for any $\varepsilon > 0$, there exists a constant $i_1 \in \mathbb{N}$ 
such that $\sup_{i \geq i_1}\{\|\Delta K_i\|_F^2 \} <  \beta_2^{-1} \cdot \varepsilon /2 $.
From \eqref{ISS_1}, one has
\begin{equation}\medmath{
    \begin{aligned} 
        \lVert \hat{P}^{i} - {P}  \rVert_F   
        &\leq  \beta_1(i -i_1 +1)  \lVert \hat{P}^{i_1} - {P}  \rVert_F  
        + \beta_2 \cdot \sup_{i \geq i_1}  \{ \lVert {\Delta K}  \rVert_{F}^2\}. \\
    \end{aligned}} \nonumber
\end{equation}
Furthermore, since $\beta_1(i) $ tends to 0 as $i \to \infty$ and
$\|\hat{P}^{i_1} - P\|_F$ is a constant independent of the step $i$, 
there exists a constant $i_2 > i_1$ 
such that $\beta_1(i -i_1 +1) \|\hat{P}^{i_1} - P\|_F < \varepsilon / 2$ holds for all $i \geq i_2$.
Then, we have $\|\hat{P}^i - P \|_F < \varepsilon.$
{{\hfill{$\Box$}}}    

    \begin{thm}  \label{thm11}  
Given $\varepsilon > 0$ and $\hat{\bar K}^1 \in \bar{\mathcal{G}}$,
where $\bar{\mathcal{G}} \triangleq \{ \hat{\bar K}^i |\hat{\bar K}^i \, \text{is a stabilizer} , Tr(\hat{S}^i) - Tr(S) \leq \varepsilon \}$,  
there exists $\delta >0$ such that if $\lVert \Delta \bar K \rVert^2_{\infty} < \delta$,
then we have
\begin{enumerate}
    \item[(i)] the following small-disturbance ISS holds:
\begin{equation}
    \begin{aligned} 
        \lVert \hat{S}^{i} - {S}  \rVert_F   
        &\leq  \bar \beta_1(i)  \lVert \hat{S}^{1} - {S}  \rVert_F  
        + \bar \beta_2  \lVert {\Delta \bar K}  \rVert_{\infty}^2. \\
    \end{aligned} \label{ISS_2_1}
\end{equation}
    \item[(ii)]  $\lim_{i \to \infty} \|\Delta  \bar K_i \|_F = 0 \text{ implies } \lim_{i \to \infty} \|\hat{S}_i - S \|_F = 0.$
\end{enumerate}
\end{thm}

   \textbf{Sketch of the proof: }
    The proof of Theorem \ref{thm11} is similar to that of Theorem \ref{thm10}.
    Using the assumption that $\hat{\bar K}^0$ is a stabilizer,
    by analyzing the Lyapunov equations satisfied by the difference matrices $\hat{S}^{i-1} - S$ and $\hat{S}^{i-1} - \hat{S}^{i}$, 
    the perturbed gain $\hat{\bar K}^i$ is stablizer when $\|\Delta  \bar K_i \|_F$ is sufficiently small.
    By taking traces and applying trace inequalities, one has  
    \begin{equation}\medmath{
    \begin{aligned} 
        Tr (\hat{S}^{i} - {S}  )  \leq (1 - \frac{1}{\bar a_1}) Tr(\hat{S}^{i-1} - {S})  + \bar a_0(\varepsilon) \cdot \bar a_2(\varepsilon) \cdot \lVert {\Delta \bar K}_i  \rVert_F^2.  
    \end{aligned}} \nonumber
    \end{equation}
    Iterating the inequality further yields the ISS type bound \eqref{ISS_2_1} for $\lVert \hat{S}^{i} - {S}  \rVert_F$.
    If $\|\Delta  \bar K_i \|_F \to 0$, the first term converges to zero and the second term in the recursive inequality vanishes asymptotically.
    Therefore, $\lim_{i \to \infty} \|\hat{S}_i - S \|_F = 0.$  
    {{\hfill{$\Box$}}}

\section{Simulation Example}
In this section, a numerical example is given to validate the effectiveness of the proposed algorithm. 

\subsection{The finite-horizon case}

Consider a system with 100 agents.  Each agent's dynamics are as follows: $Q=diag(1,1,1),R=1,\Gamma=diag(0.9,0.9,0.9),T=6, \sigma =1$,
$B=\left[  
    -0.2,
      0, 
     0.1
 \right]^T,$
$D=\left[ 
    -0.02,
     -0.03,
     0.01
 \right]^T $,
$ K_0 =\left[ 
    0.1, -0.1, -0.7
 \right]^T $,
$e_j = sin^2(0.5k)+ sin(k) + \xi \cdot cos(k),$ with
$\xi {\sim} \mathcal{N} (0,1),$
\begin{equation}
A= \begin{bmatrix}
    0.6&  0.2& -0.2 \\
    0.2&  0.6& -0.2 \\
    0.4& 0.5& 0.6 \\
\end{bmatrix},
C= \begin{bmatrix}
    0.06&  0.02& -0.02\\
     0.02& 0.06& -0.02\\
     0.04& 0.05& 0.06 \\
\end{bmatrix}.\\
\nonumber
\end{equation}
The simulation results for finite-horizon in 3-D case are plotted in Figures \ref{fig_all_finite}. 
After a period of time, they converge to a constant.
Therefore, the simulation results demonstrate
the effectiveness of the proposed finite-horizon mean field model-free algorithm.

\begin{figure}
	\centering
	\subfigure[]{
		\begin{minipage}[t]{0.48\linewidth}
			\centering
			\includegraphics[width=1.7in]{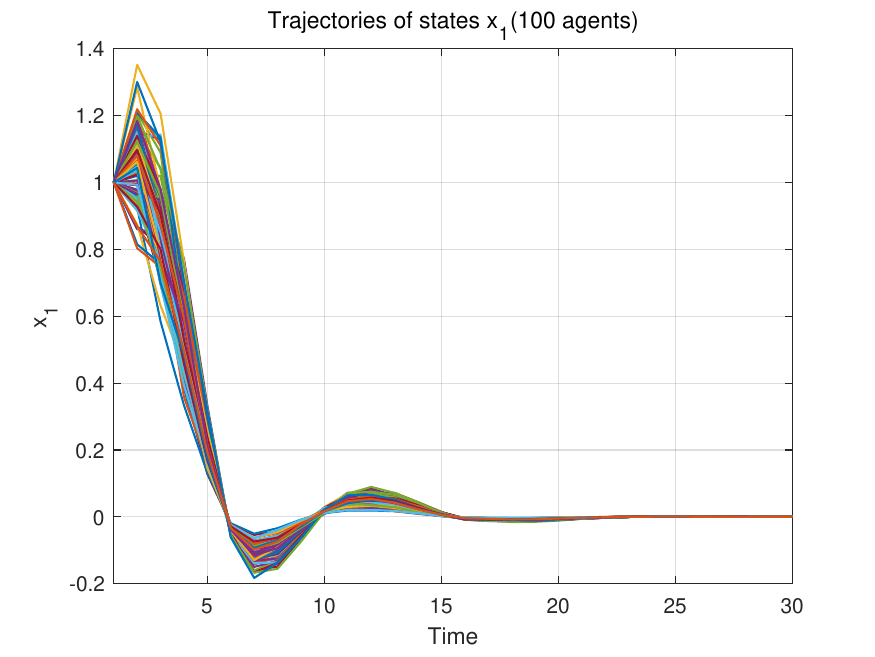}
		\end{minipage}
	}%
	\subfigure[]{
		\begin{minipage}[t]{0.48\linewidth}
			\centering
			\includegraphics[width=1.7in]{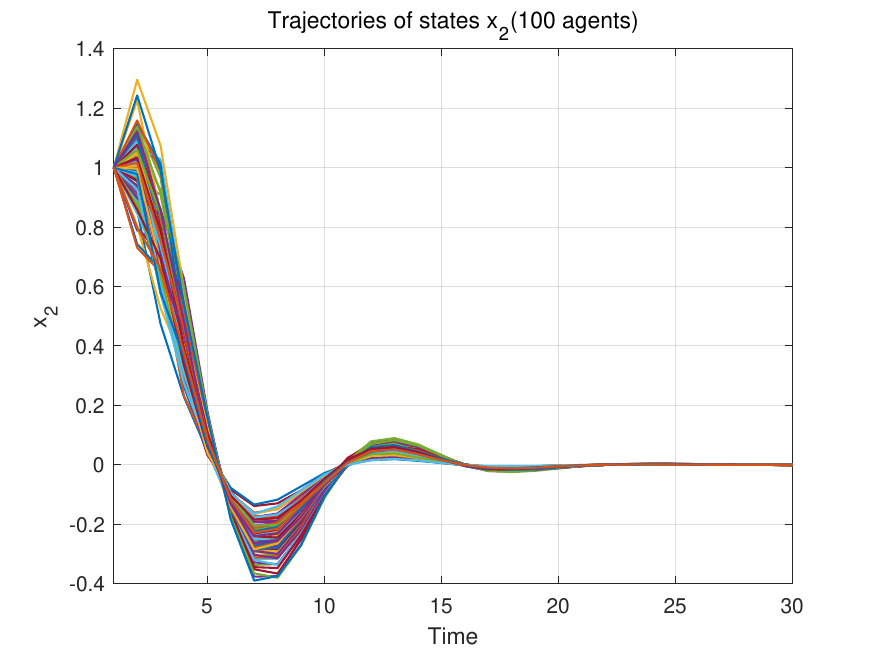}
		\end{minipage}
	}%

	\subfigure[]{
		\begin{minipage}[t]{0.48\linewidth}
			\centering
			\includegraphics[width=1.7in]{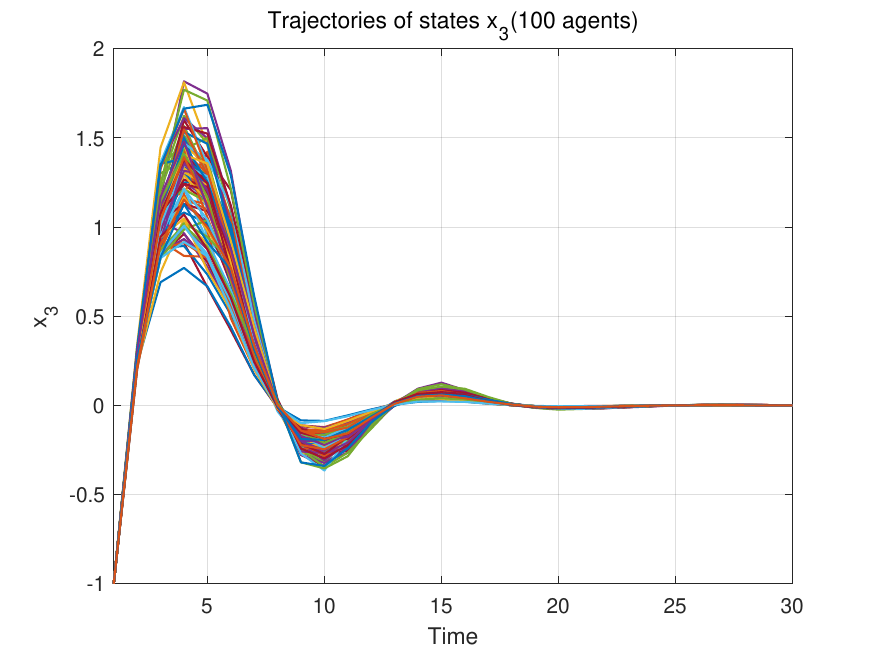}
		\end{minipage}
	}%
	\subfigure[]{
		\begin{minipage}[t]{0.48\linewidth}
			\centering
			\includegraphics[width=1.7in]{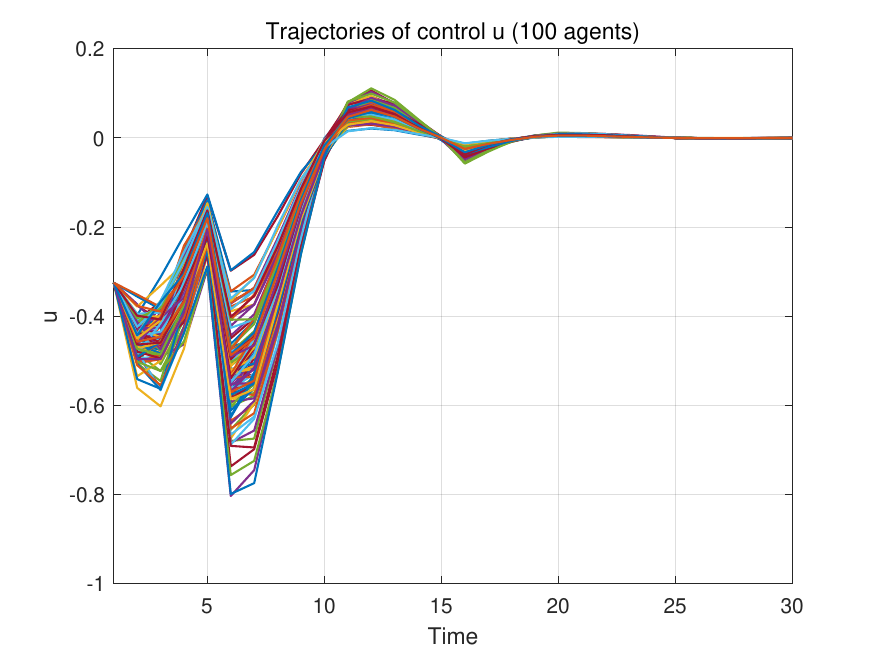}
		\end{minipage}
	}%
	\centering
	\caption{Simulation results for finite-horizon in 3-D case (100 agents). 
     (a) The trajectories of system states $x_1$.
     (b) The trajectories of system states $x_2$.
     (c) The trajectories of system states $x_3$.
     (d) Sample paths of control $u$. }
	\label{fig_all_finite}
\end{figure}

\subsection{ The infinite-horizon case}
Consider a system with 10 agents.  Each agent's dynamics are as follows: 
$R=0.99,\Gamma=diag(0.94,0.94),  \sigma =1$,
$B=\left[  
    -0.84 ,
    -0.67 
 \right]^T,$
$D=\left[ 
    0.03 ,
    0.07 
 \right]^T $,
$ K_0 =\left[ 
    -1.33  ,   3.42 
 \right]^T $,
$e_j = sin^2(0.5k)+ sin(k) + \xi \cdot cos(k)$, with 
$\xi {\sim} \mathcal{N} (0,1),$
\begin{equation} \medmath{
A= \begin{bmatrix}   
    0.65 &  -0.72  \\  
    -0.11 & 0.57
\end{bmatrix},
C= \begin{bmatrix}    
    0.01 &  0.02 \\
    -0.04 & 0.09
\end{bmatrix} 
Q= \begin{bmatrix} 
    0.04 &  -0.86 \\
    -0.86 & 0.91
\end{bmatrix}.\\
}\nonumber
\end{equation}

\begin{figure}  
	\centering
	\subfigure[]{
		\begin{minipage}[t]{0.3\linewidth}
			\centering
			\includegraphics[width=1.19in]{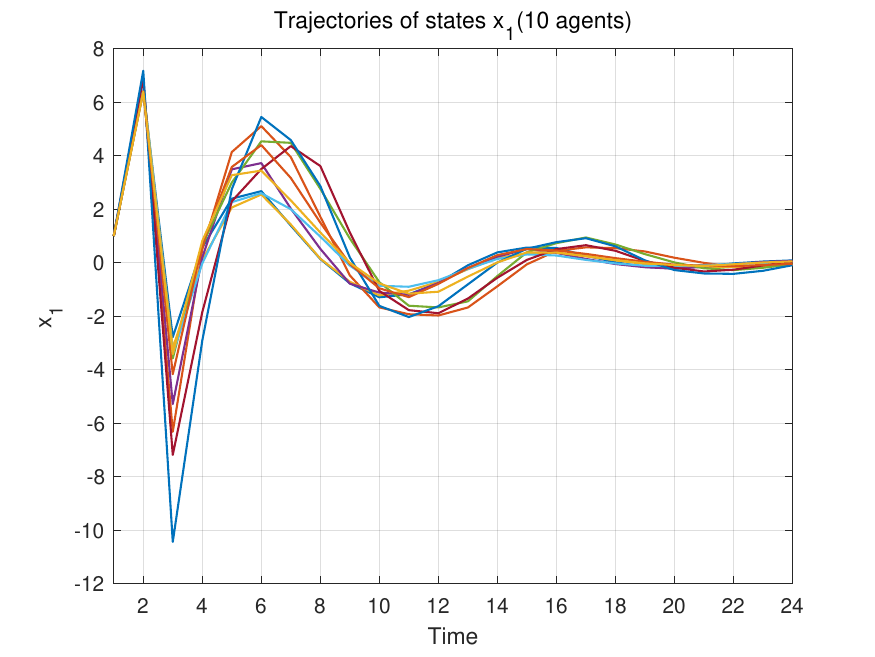}
		\end{minipage}
	}%
	\subfigure[]{
		\begin{minipage}[t]{0.3\linewidth}
			\centering
			\includegraphics[width=1.19in]{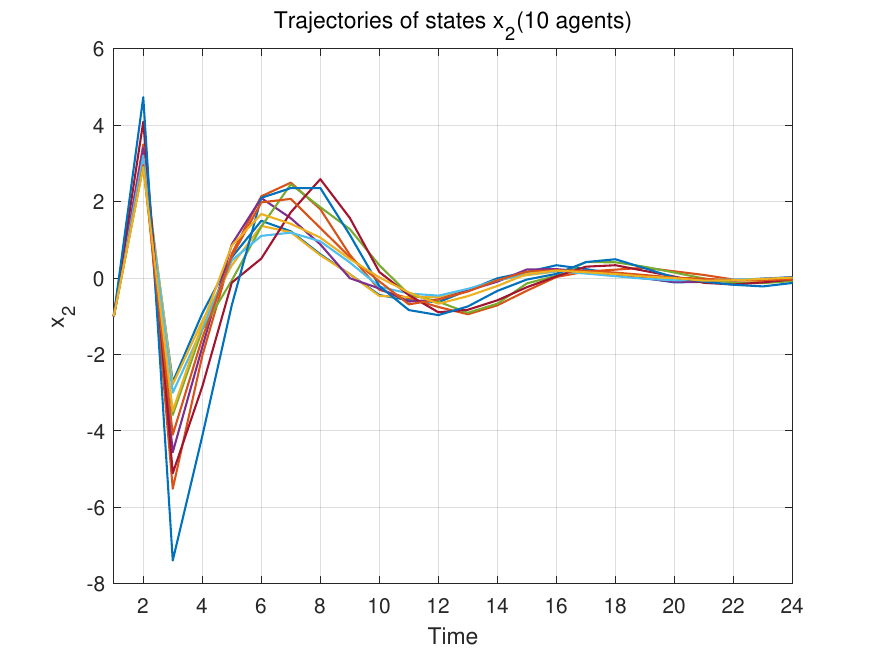}
		\end{minipage}
	}%
	\subfigure[]{
		\begin{minipage}[t]{0.3\linewidth}
			\centering
			\includegraphics[width=1.19in]{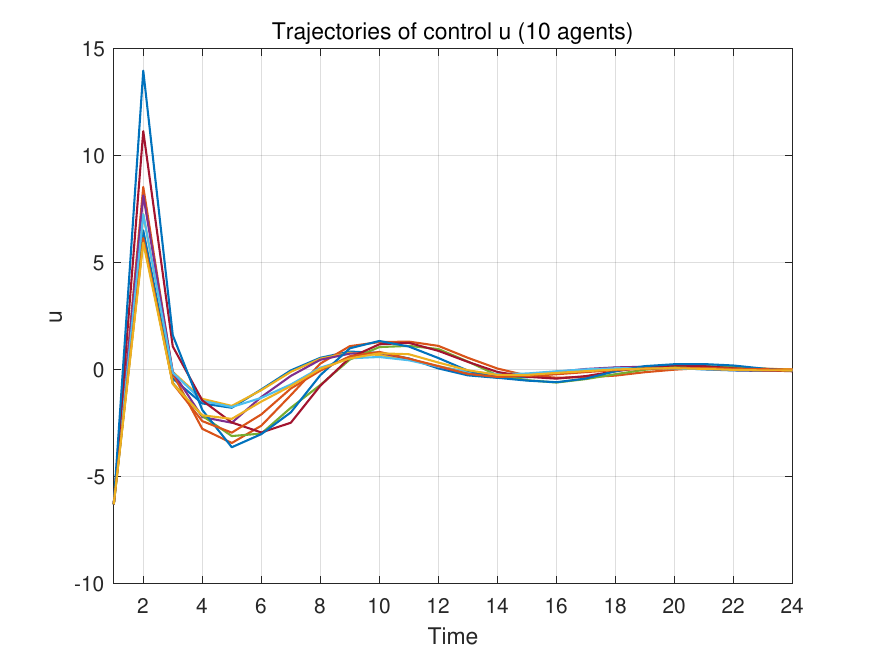}
		\end{minipage}
	}%
	\centering
	\caption{Simulation results for infinite-horizon in 2-D case (10 agents). 
     (a) The trajectories of system states $x_1$.
     (b) The trajectories of system states $x_2$.
     (c) Sample paths of control $u$. }
	\label{fig_all_infinite}
\end{figure}

\begin{figure}[t]
	\centering
	\includegraphics[width=0.9\linewidth]{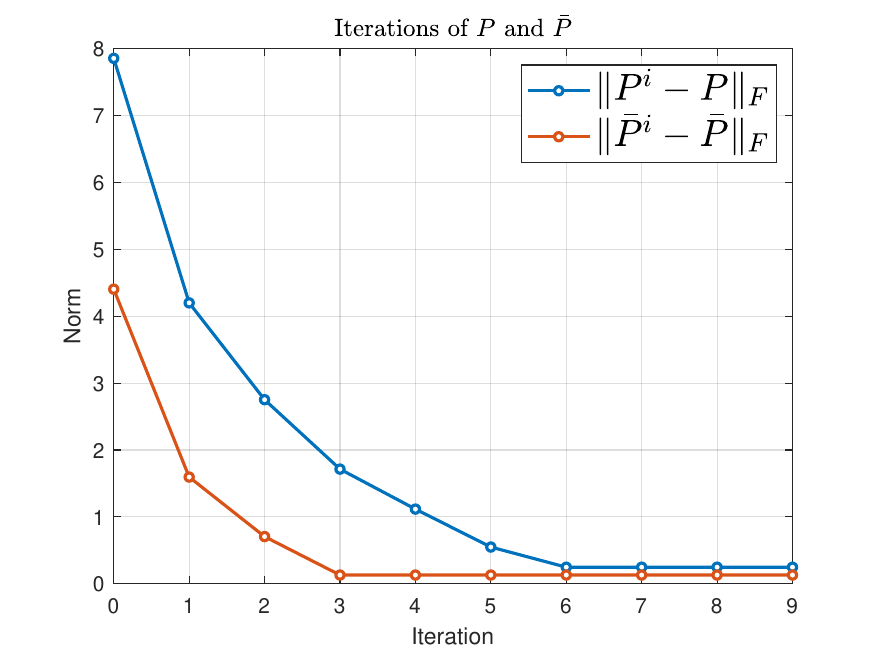}
	\caption{Iterations of $P$ and $\bar{P}$ parameters in infinite-horizon case.}
	\label{fig_infinite_P}
\end{figure}
 
\begin{figure}
	\centering
	\subfigure[$\medmath{(\Delta A,\Delta B)= ( -0.2 A, -0.2B)}$]{  
		\begin{minipage}[t]{0.48\linewidth}
			\centering
			\includegraphics[width=1.7in]{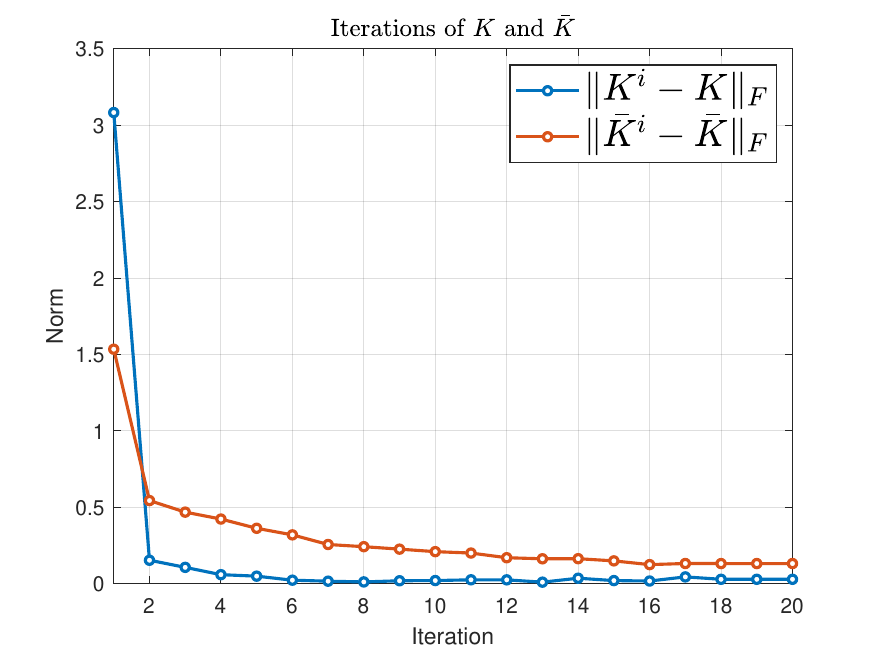}
		\end{minipage}
	}%
	\subfigure[$\medmath{(\Delta A,\Delta B)= ( -0.5 A, -0.5B)}$]{  
		\begin{minipage}[t]{0.48\linewidth}
			\centering
			\includegraphics[width=1.7in]{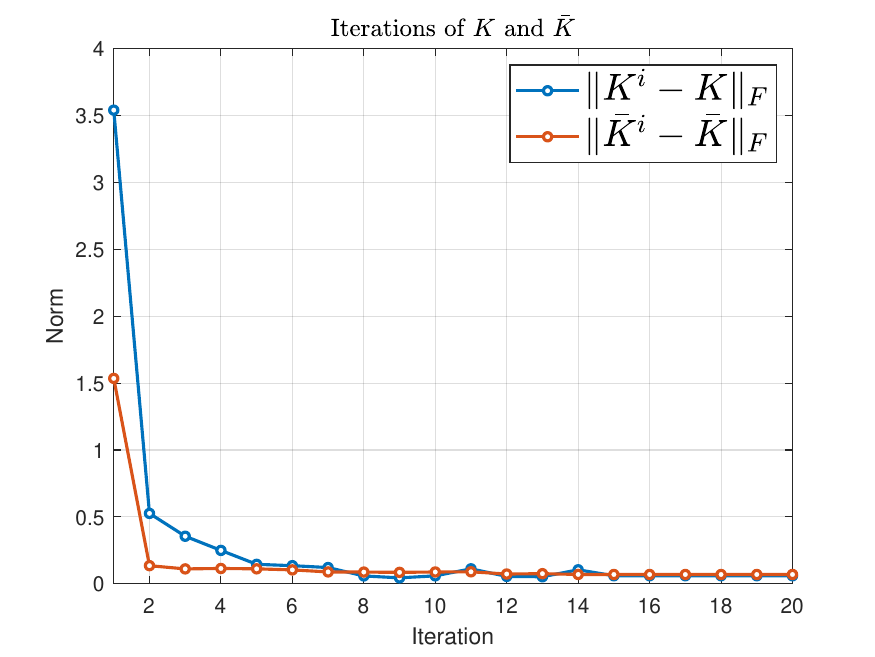}
		\end{minipage}
	}%

	\subfigure[$\medmath{(\Delta C,\Delta D)= ( +0.3 C, +0.3D)}$]{ 
		\begin{minipage}[t]{0.48\linewidth}
			\centering
			\includegraphics[width=1.7in]{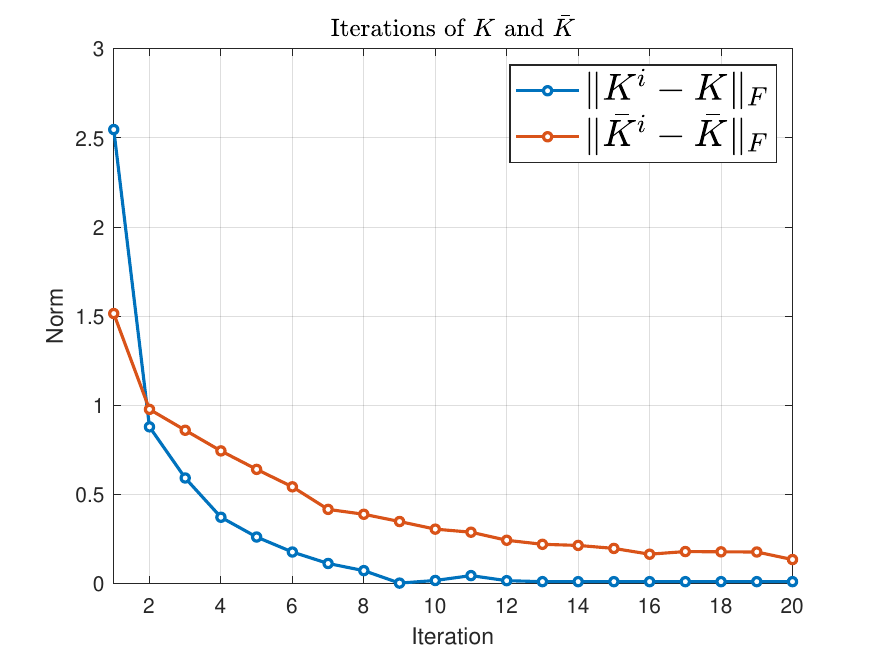}
		\end{minipage}
	}%
	\subfigure[$\medmath{(\Delta C,\Delta D)= ( -0.2 C, -0.2D)}$]{   
		\begin{minipage}[t]{0.48\linewidth}
			\centering
			\includegraphics[width=1.7in]{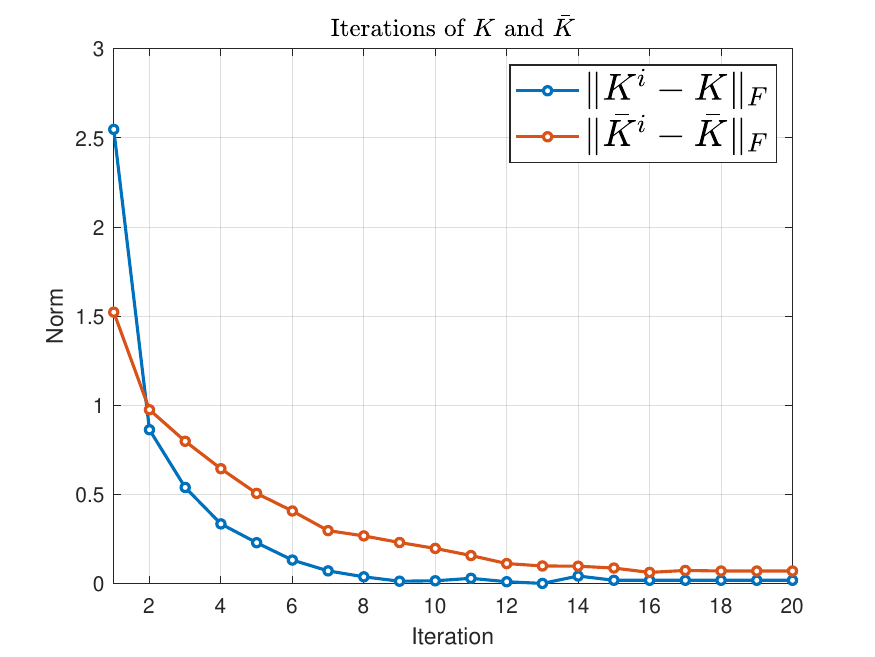}
		\end{minipage}
	}%

	\subfigure[$N=5$]{
		\begin{minipage}[t]{0.48\linewidth}
			\centering
			\includegraphics[width=1.7in]{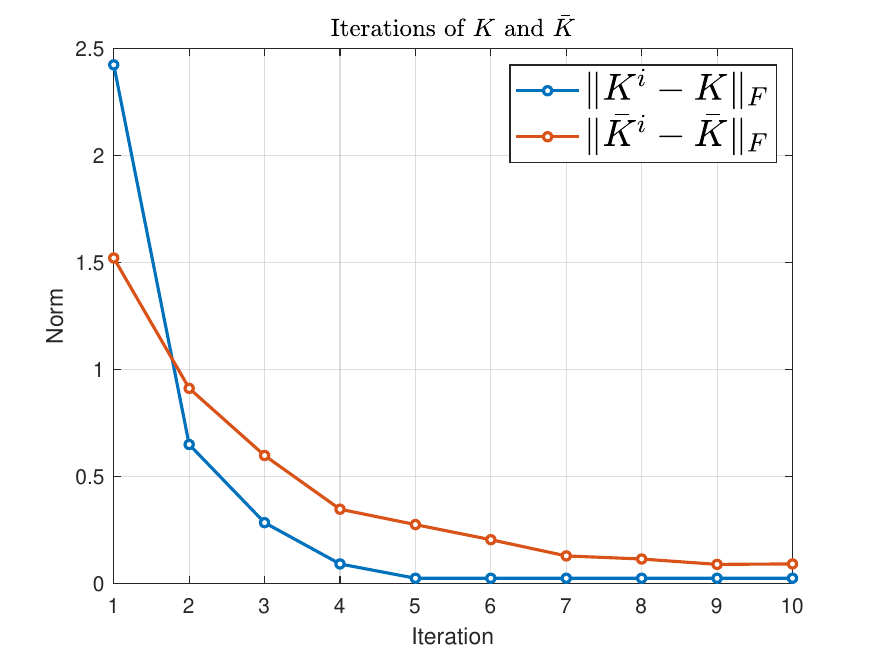}
		\end{minipage}
	}%
	\subfigure[$N=10000$]{
		\begin{minipage}[t]{0.48\linewidth}
			\centering
			\includegraphics[width=1.7in]{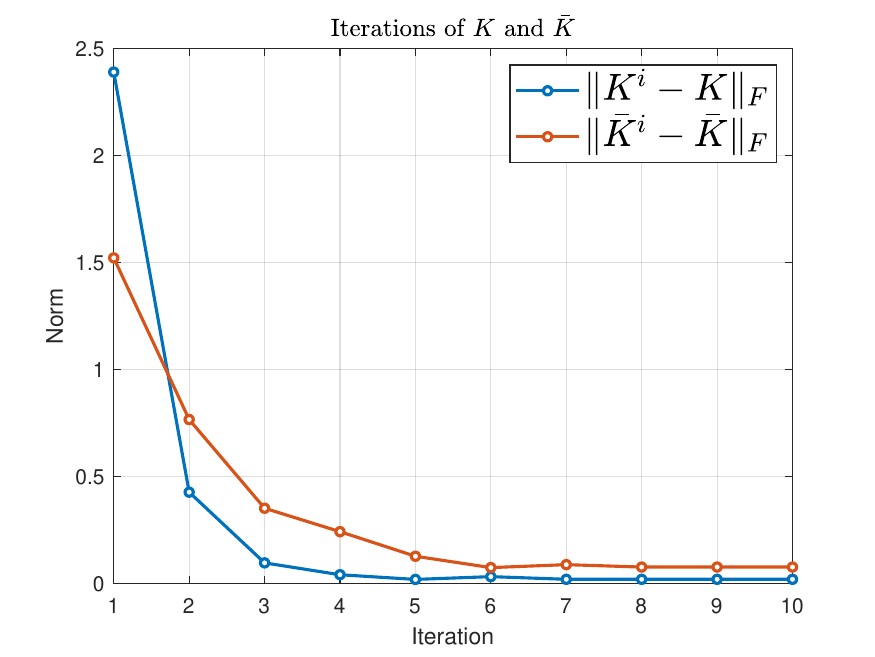}
		\end{minipage}
	}%
	\centering
	\caption{ Gain matrices errors under different perturbations of $(A, B, C, D)$ and $N$.} 
    \label{sensitivity_grid_infinite_horizon}
\end{figure}

The eigenvalues of $Q$ are $-0.49$ and $1.43$ respectively, 
in other words, $Q$ is indefinite.
By using the model-free algorithm, we obtain the solution of the AREs \eqref{section4_1}-\eqref{section4_2}: \\
{
\begin{equation}  
P= \begin{bmatrix}      
    -0.3905 & -0.4531   \\
    -0.4531 & 0.8474
\end{bmatrix},
\bar{P}= \begin{bmatrix}          
    0.3784  &  0.4801 \\
    0.4801  & -0.8466    
\end{bmatrix}.\nonumber
\end{equation}}
The optimal control gains are \\
$K=\left[ \begin{matrix}  
    -0.7602  &  0.9381
\end{matrix} \right],
\bar{K}=\left[ \begin{matrix} 
    -0.0381  & -0.0025  
\end{matrix} \right].$
 
The simulation results for infinite-horizon in 2-D case are plotted in Figures \ref{fig_all_infinite}. 
After a period of time, they converge to a constant.
The resulting errors are $\|P^i - P^*\|_F = 0.1123$ and $\|\bar{P}^i - \bar{P}^*\|_F = 0.0641$.
The iterations of $P$ and  $\bar P$ are presented in Figures \ref{fig_infinite_P}. 
This indicates that the iteartion defined by  \eqref{infinite_1_model_free_theorem} and \eqref{infinite_2_model_free_theorem}  converges.
The results of the sensitivity analysis are presented in Figure \ref{sensitivity_grid_infinite_horizon},
which shows the robustness when perturbing parameters $(A, B, C, D)$ and population sizes.
Therefore, the simulation results demonstrate
the effectiveness of model-free algorithm for infinite-horizon case.

\section{Conclusion}
In this paper, a novel computational model-free algorithm has been presented 
to learn decentralized strategies for discrete-time finite-population systems.
The weights $Q$ and $R$ in the cost functional are not limited to be positive semidefinite. 
The model-free algorithm is proposed to learn the optimal policies directly
from trajectory data,  without learning the system parameters.
For the finite-horizon and infinite-horizon cases, 
the convergence of algorithms is separately established, and the obtained RL solutions are compared.
Simulation results demonstrate the effectiveness of the proposed algorithm.

\begin{ack}                                
The authors would like to thank the editor and three anonymous referees 
for  their constructive and insightful comments  for improving the quality of this work.
\end{ack}

\bibliographystyle{automatica}         
\bibliography{my_meanfield}

\end{document}